\documentclass[11pt]{amsart}
\usepackage[latin1]{inputenc}
\usepackage{amsmath}
\usepackage{amsfonts}
\usepackage{amssymb}
\usepackage{graphicx}
\usepackage{amsthm}
\usepackage{amssymb}
\usepackage{cite}
\usepackage{mathrsfs}
\usepackage{hyperref}
\usepackage{paralist}
\usepackage{a4wide}
\usepackage{enumitem}
\usepackage{tabularx}
\begin{document}

\newcommand{\B}{\mathbf{B}}
\newcommand{\U}{\mathbf{U}}
\newcommand{\T}{\mathbf{T}}
\newcommand{\G}{\mathbf{G}}
\newcommand{\Para}{\mathbf{P}}
\newcommand{\Levi}{\mathbf{L}}

\newcommand{\Gtilde}{\mathbf{\tilde{G}}}
\newcommand{\Ttilde}{\mathbf{\tilde{T}}}
\newcommand{\Btilde}{\mathbf{\tilde{B}}}

\newcommand{\N}{\operatorname{N}}
\newcommand{\Z}{\operatorname{Z}}
\newcommand{\Gal}{\operatorname{Gal}}
\newcommand{\Kerel}{\operatorname{ker}}
\newcommand{\Irr}{\operatorname{Irr}}
\newcommand{\D}{\operatorname{D}}
\newcommand{\I}{\operatorname{I}}
\newcommand{\GL}{\operatorname{GL}}
\newcommand{\SL}{\operatorname{SL}}
\newcommand{\W}{\operatorname{W}}
\newcommand{\R}{\operatorname{R}}

\theoremstyle{remark}

\theoremstyle{definition}
\newtheorem{definition}{Definition}[section]
\newtheorem{construction}[definition]{Construction}
\newtheorem{remark}[definition]{Remark}
\newtheorem{example}[definition]{Example}
\newtheorem{notation}[definition]{Notation}

\theoremstyle{plain}
\newtheorem{theorem}[definition]{Theorem}
\newtheorem{lemma}[definition]{Lemma}
\newtheorem{proposition}[definition]{Proposition}
\newtheorem{corollary}[definition]{Corollary}
\newtheorem{conjecture}[definition]{Conjecture}
\newtheorem{assumption}[definition]{Assumption}
\newtheorem{main theorem}[definition]{Main Theorem}

\title{The Navarro refinement of the McKay conjecture for finite groups of Lie type in defining characteristic}
\date{\today}
\author{Lucas Ruhstorfer}
\address{Fachbereich Mathematik, TU Kaiserslautern, 67653 Kaiserslautern, Germany}
\email{ruhstorfer@mathematik.uni-kl.de}


\begin{abstract}
In this paper we verify Navarro's refinement of the McKay conjecture for quasi-simple groups of Lie type in their defining characteristic. Navarro's refinement takes into account the action of specific Galois automorphisms on the characters present in the McKay conjecture \cite{Navarro}. Our proof of this case of the conjecture relies on a character correspondence constructed by Maslowski in \cite{Maslowski2010}. Building on this we verify the inductive condition for Navarro's refinement from \cite{NSV} for most groups of Lie type in defining characteristic. 
\end{abstract}

\maketitle

\section{Introduction}

For a finite group $G$, a prime $p$ and a Sylow $p$-subgroup $P$ of $G$ the McKay conjecture asserts that there exists a bijection between the set of $p'$-degree characters $\Irr_{p'}(\N_G(P))$ of $\N_G(P)$ and the set of $p'$-degree characters $\Irr_{p'}(G)$ of $G$.

However, Navarro suggests that there should exist a bijection between these sets of characters which is compatible with certain Galois automorphisms. Denote by $\mathbb{Q}^{\mathrm{ab}}$ the field generated by all roots of unity in $\mathbb{C}$. Let $\mathcal{H}$ be the subgroup of $\mathcal{G}:=\mathrm{Gal}(\mathbb{Q}^{\mathrm{ab}} / \mathbb{Q})$ consisting of all $\sigma\in \mathcal{G}$ for which there exists an integer $e$ such that $\sigma$ sends any $p'$-root of unity $\zeta$ to $\zeta^{p^e}$. He then proposes the following refinement of the McKay conjecture, see \cite[Conjecture A]{Navarro}.

\begin{conjecture}\label{Navarro}
There exists an $\mathcal{H}$-equivariant bijection between $\Irr_{p'}(\N_G(P))$ and $\Irr_{p'}(G)$.
\end{conjecture}

This conjecture implies several character-theoretical consequences. One of them was proved by Navarro, Tiep and Turull, see \cite[Theorem A]{Turull} and another recently by Schaeffer Fry, see \cite{SchaefferFry}.

It therefore seems important to study and verify Conjecture \ref{Navarro} for as many families of finite groups as possible. We contribute to this program by proving the following theorem.

\begin{theorem}\label{maintheorem}
Let $\G$ be a simple algebraic group of simply connected type defined over an algebraic closure of $\mathbb{F}_p$ and $F: \G \to \G$ a Frobenius endomorphism. Suppose that $(\G,F)$ is not contained in the table of Theorem \ref{bijection} below. Then there exists an $\mathcal{H}$-equivariant bijection
	\begin{align*}
	f:\Irr_{p'}(\mathrm{N}_{\G^F}(P) )\to \Irr_{p'}(\G^F),
	\end{align*}
where $P$ is a Sylow $p$-subgroup of $\G^F$.
%	Moreover, for every central character $\lambda\in\Irr(\Z(\G^F))$ the map $f$ restricts to a bijection $\Irr_{p'}(\B^F\mid \lambda)^\sigma\to \Irr_{p'}(\G^F\mid\lambda)^\sigma$.
\end{theorem}

%Observe that in the situation of the theorem, $\B^F$ is exactly the normalizer of a Sylow $p$-subgroup of $\G^F$.
%We also remark that we can prove Theorem \ref{maintheorem} for bad primes $p$ if the center of $\G^F$ is trivial, see Corollary \ref{better} for a precise statement.
The proof of Theorem \ref{maintheorem} is based on Maslowski's work on the inductive McKay condition for simple groups of Lie type in defining characteristic. He constructs a certain automorphism-equivariant bijection for the $p'$-characters of the universal covering group of a finite simple group of Lie type defined over a field of characteristic $p$.

For the original McKay conjecture a reduction theorem was proved by Isaacs, Malle and Navarro, see \cite[Theorem B]{Isaacs}. Recently, Navarro, Späth and Vallejo \cite{NSV} were able to provide a reduction theorem for Navarro's refinement of the McKay conjecture. Their theorem asserts that Conjecture \ref{Navarro} holds for all finite groups and the prime $p$ if all nonabelian simple groups satisfy the so-called inductive Galois--McKay condition for the prime $p$. We show here that the inductive Galois--McKay condition holds for many groups of Lie type in defining characteristic. 

\begin{theorem}\label{inductive}
Suppose that $(\G,F)$ satisfies Assumption \ref{spaethassumption}. If $S:=\G^F/ \mathrm{Z}(\G^F)$ is a simple non-abelian group and $\G^F$ is its universal covering group then the inductive Galois--McKay condition from \cite[Definition 3.1]{NSV} holds for the group $S$ and the prime $p$.
\end{theorem}

To prove this theorem, we show that the bijection constructed in Theorem \ref{maintheorem} is suitable for the inductive Galois--McKay condition. In order to do this, we compute the stabilizers of $p'$-characters under the simultaneous action of Galois automorphisms and group automorphisms. Then we use the theory of Gelfand--Graev characters for disconnected reductive groups to explicitly compute extensions of these characters to certain almost simple groups. The results obtained here may be of independent interest since they give information about character values of characters of almost simple groups.
% As said in the Introduction of \cite{NSV} the verification of the inductive Galois--McKay condition for finite simple groups brings the  new challenge, as it requires a vast knowledge of the character values ofdecorated simple groups
%Later Späth used Maslowski's result to prove that these groups satisfy the inductive McKay conditions for the prime $p$, see \cite[Theorem 1.1]{Spaeth}.
%
%The verification of the inductive Galois?McKay condition for finite simple groupsbrings up a new challenge, as it requires a vast knowledge of the character values ofdecorated simple groups and the interplay between Galois action and the action ofgroup automorphisms on characters, a subject that it is still not fully understood.

The structure of our paper is as follows. In Section \ref{notation} we introduce some notation. In Section \ref{rep} we recall some basic facts about the representation theory of finite groups of Lie type. We discuss the action of Galois automorphisms on Lusztig series and recall a description of irreducible $p'$-characters originally due to Green, Lehrer and Lusztig. In Section \ref{Maslowski} we recall the McKay bijection due to Maslowski\cite{Maslowski2010}. In Section \ref{section5} we use the results of the two previous sections to prove Theorem \ref{maintheorem}. Section \ref{section6} and \ref{section7} are then devoted to the proof Theorem \ref{inductive}.
%We will finally discuss the difficulties in the remaining cases.

\section*{Acknowledgement}

This paper originates from the results of the author's master thesis \cite{Ruhstorfer} at the Technische Universität Kaiserslautern. I would like to express my gratitude to my supervisor Gunter Malle for his useful remarks and comments during the development of my thesis. I thank Britta Späth for suggesting this topic and for fruitful discussions.

The author would like to thank the Isaac Newton Institute for Mathematical Sciences for support and hospitality during the programme Groups, representations and applications: new perspectives when work on this paper was undertaken. This work was supported by: EPSRC grant number EP/R014604/1.

\section{Notation}\label{notation}

\subsection{Rings and fields}\label{rings}

For an integer $m$ we denote by $\mathbb{Q}_m$ the $m$-th cyclotomic field.
Let $p$ be a prime and $q$ an integral power of $p$. We let $\mathbf{k}$ be an algebraic closure of $\mathbb{F}_p$. Let $\ell$ be a prime different from $p$ and denote by $K$ an algebraic closure of the field of $\ell$-adic numbers. Denote by $(\mathbb{Q}/\mathbb{Z})_{p'}$ the subgroup of elements of the abelian group $\mathbb{Q}/\mathbb{Z}$ whose order is not divisible by $p$. We fix once and for all an isomorphism $\mathbf{k}^\times \to (\mathbb{Q}/\mathbb{Z})_{p'}$ and an injective morphism $\mathbf{k}^\times \hookrightarrow K^\times$.

\subsection{Characters of finite groups}

If $Y$ is a finite group we denote by $\Irr(Y)$ the set of irreducible $K$-valued characters of $Y$. For $K$-valued characters we mostly follow the notation of \cite{Isaacs2013}. Let us briefly state the main deviations from the notation in \cite{Isaacs2013}. If $X$ is a normal subgroup of $Y$ and $\vartheta \in \Irr(X)$ we denote by $Y_\vartheta$ the inertia group of $\vartheta$ in $Y$. Moreover, we denote by $\Irr(Y \mid \vartheta)$ the set of irreducible characters of $Y$ which occur as constituents of the induced character $\vartheta^Y$. We say that characters in the set $\Irr(Y \mid \vartheta)$ lie above $\vartheta$. Similarly, if $\chi \in \Irr(Y)$ we mean by $\Irr(X \mid \chi)$ the set of irreducible characters of $X$ occurring as constituents of the restriction $\chi_X$. Such characters are said to lie below $\chi$.

\subsection{Finite groups and Galois automorphisms}

Let $Y$ be a finite group.
%and $m:=|Y|$.
By Brauer's theorem the Galois group $\mathcal{G}= \mathrm{Gal}(\mathbb{Q}^{\mathrm{ab}} / \mathbb{Q})$ acts on the set of irreducible characters $\Irr(Y)$. Following the notation of \cite{NSV} for a Galois automorphism $\sigma \in \mathcal{G}$ and a generalized character $\chi \in \mathbb{Z} \Irr(Y)$ we let $\chi^\sigma \in \mathbb{Z} \Irr(Y)$ be the generalized character defined by $\chi^\sigma(y)=\sigma(\chi(y))$, for $y\in Y$. This defines indeed a group action of $\mathcal{G}$ on $\Irr(G)$ since $\mathcal{G}$ is abelian.
 Furthermore, $\chi^{\mathcal{H}}$ denotes the $\mathcal{H}$-orbit $\{ \chi^\sigma \mid \sigma \in \mathcal{H} \}$ of the character $\chi$.

\section{Representation theory of groups of Lie type}\label{rep}

\subsection{Lusztig series and Galois automorphisms}\label{sub}

Let $\G$ be a connected reductive group defined over $\mathbb{F}_q$ via a Frobenius endomorphism $F:\G \to \G$. We fix an $F$-stable maximal torus $\T$ of $\G$ contained in an $F$-stable Borel subgroup $\B$. Let $\U$ be the unipotent radical of $\B$. We denote by $\Phi$ the root system of $\G$ with respect to the torus $\T$ and by $\Delta=\{\alpha_1,\dots,\alpha_n\}$ the set of simple roots of $\Phi$ with respect to $\T \subseteq \B$. We let $\Phi^+$ be the set of positive roots and $\Phi^\vee$ the set of coroots.

Fix a triple $(\G^*,\T^*,F^*)$ in duality with $(\G,\T,F)$ as in \cite[Definition 13.10]{Digne1991}. This together with the choices made in \ref{rings} gives rise to a bijection between the set of $\G^F$-conjugacy classes of pairs $(\mathbf{S},\theta)$ where $\mathbf{S}$ is an $F$-stable maximal torus of $\G$ and $\theta\in\Irr(\mathbf{S}^F)$ with the set of $(\G^*)^{F^*}$-conjugacy classes of pairs $(\mathbf{S}^*,s)$ where $s\in (\G^*)^{F^*}$ is a semisimple element and $\mathbf{S}^*$ is an $F^*$-stable maximal torus with $s\in \mathbf{S}^*$, see \cite[Proposition 13.13]{Digne1991}.

%Let $\mathbf{S}$ be a maximal $F$-stable torus of $\G$. We denote by $R_{\mathbf{S}}^{\G}(\theta)$ the Deligne--Lusztig character associated to the character $\theta \in \Irr(\mathbf{S}^F)$, see \cite[Definition 11.1]{Digne1991}.

%As usually, we write $R_{{\mathbf{S}}^*}^\G(s)$ for the character $R_{\mathbf{S}}^\G(\theta)$ if $(\mathbf{S},\theta)$ is in duality with $({\mathbf{S}}^*,s)$.

If $s\in ( \G^\ast)^{F^\ast}$ is a semisimple element we denote by $(s)$ its $ ( \G^\ast)^{F^\ast}$-conjugacy class. We denote by $\mathcal{E}(\G^F,(s)) \subseteq \Irr(\G^F)$ its rational Lusztig series. We have the following lemma, see the proof of \cite[Lemma 9.1]{Tiep}.
 
\begin{lemma}\label{sigmadl} Let $\sigma\in \mathcal{G}$ such that and $\sigma(\zeta)=\zeta^k$ for a primitive $|\G^F|$-th root of unity $\zeta \in K$. Then we have $\sigma(\mathcal{E}(\G^F,(s)))=\mathcal{E}(\G^F,(s^k))$.
%	Then $R_{\mathbf{S}}^\G(\theta)^\sigma=R_{\mathbf{S}}^\G(\theta^k)$ for any $\theta\in \Irr(\mathbf{S}^F)$. 
\end{lemma}

%To each semisimple conjugacy class $(s)$ of $( \G^\ast)^{F^\ast}$ we associate a character $\chi_{(s)}\in \mathbb{Z}_{\geq 0} \Irr(\G^F)$, as in \cite[Definition 14.40]{Digne1991}.
%If the center of $\G$ is connected then $\chi_{(s)}$ is an irreducible character by \cite[Corollary 14.47(a)]{Digne1991} and we have $\chi_{(s)}\in \mathcal{E}(\G^F,(s))$.

%given by
%$$D_\G= \displaystyle \sum_{\mathbf{P}} (-1)^{\text{rank}(\mathbf{P})} R_{\mathbf{L}}^{\G} {}^* R_{\mathbf{L}}^{\G},$$ 
%where the summation is over all rational parabolic subgroups $\mathbf{P}$ of $\G$ containing $\B$ and $\mathbf{L}$ is an $F$-stable Levi subgroup of $\mathbf{P}$

\subsection{Gelfand--Graev characters}\label{Gelfandsection}

In order to introduce the Gelfand--Graev characters of $\G^F$ we proceed as in the proof of \cite[Theorem 2.4]{Digne92}. The Frobenius endomorphism $F$ of $\G$ induces an automorphism $\gamma$ of the character group $X(\T)$. Since $\T$ is a maximally split torus it follows by \cite[Proposition 22.2]{Malle2011} that $\gamma$ stabilizes the set of positive roots $\Phi^+$ and the set of simple roots $\Delta$. Hence, $\gamma$ acts on the index set of $\Delta= \{ \alpha_1, \ldots, \alpha_n \}$ which yields a partition \begin{align*} \{1, \dots, n\} =A_1 \cup \dots \cup A_r\end{align*}
of the index set of $\Delta$ into its $\gamma$-orbits. For each $A_i$ we fix a representative $a_i\in A_i$. If $\alpha\in \Phi$ we let $\U_{\alpha}$ be the root subgroup of $\G$ associated to the root $\alpha\in \Phi$. We denote by $\U_{A_i}$, $i=1,\dots,r$, the product in $\U/ [\U,\U]$ of the root subgroups $\U_{\alpha}$, $\alpha_j$, $j\in A_i$. By \cite[Lemma 2.2]{Digne92} we have
\begin{align*}
\U^F / {[\U,\U]}^F \cong \prod_{i=1}^{r} \U_{A_i}^F.
\end{align*}
For each $\alpha\in\Phi$ there is an isomorphism $x_\alpha:(\mathbf{k},+)\to \U_\alpha$ with $F(x_\alpha(a))=x_{\gamma (\alpha)}(a^q)$ for all $a\in \mathbf{k}$ and all $\alpha\in\Phi$. These maps induce an isomorphism $x_i:(\mathbb{F}_{q^{|A_i|}},+)\to\U_{A_i}^F$ given by
\begin{align*} x_i(a)=\prod_{k=0}^{|A_i|-1} x_{\gamma^k \alpha_{a_i} }(a^{q^k}).\end{align*}
Now fix a character $\phi_0\in\Irr \left( (\mathbb{F}_{q^{N}},+) \right)$, where $N=\mathrm{lcm}(|A_1|, \dots, |A_r|)$, such that the restriction of $\phi_0$ to $({\mathbb{F}_q},+)$ is nontrivial. Then any character $\psi\in\Irr(\U_{A_i}^F)$ is given by $\psi(x_i(a))=\phi_0(c_i a)$ for all $a\in\mathbb{F}_{q^{|A_i|}}$ and some $c_i\in\mathbb{F}_{q^{|A_i|}}$. Any irreducible character $\phi\in\Irr \left((\U / [\U,\U])^F \right)$ is of the form $\phi=\displaystyle\prod_{i=1}^{r} \phi_i$ for some characters $\phi_i \in \Irr(\mathbf{U}_{A_i}^F)$ and so we obtain the following.

\begin{lemma}\label{characterU} The map $\delta:\Irr \left((\U / [\U,\U])^F \right)\to \displaystyle\prod_{i=1}^{r} \mathbb{F}_{q^{|A_i|}}$
given by 
$$\phi=\displaystyle\prod_{i=1}^{r} \phi_i \mapsto (c_1, \ldots ,c_r),$$
where the $c_i\in\mathbb{F}_{q^{|A_i|}}$ are such that $\phi_i(x_i(a))=\phi_0(c_i a)$ for all $a\in\mathbb{F}_{q^{|A_i|}}$ is a bijection.
\end{lemma}
Let $\iota:\G \hookrightarrow \Gtilde$ be an extension of $\G$ by a central torus such that $\Gtilde$ is a connected reductive group with connected center. Let $\xi=\delta^{-1}(1,\ldots,1)$ or more concretely,
\begin{align*}
\xi=\prod_{i=1}^{r} \phi_i \text{, with } \phi_i(x_i(a))=\phi_0(a) \text{ for } a\in\mathbb{F}_{q^{|A_i|}}.
\end{align*}
Denote by $\Gamma_1$ the induced character $\xi^{\G^F}$. We have a natural isomorphism $H^1(F,\Z(\G)) \cong \Gtilde^F /\G^F \Z(\Gtilde)^F$ by \cite[Corollary 1.2]{Lehrer1978} and \cite[Proposition 1.5]{Lehrer1978}. For $z\in H^1(F,\Z(\G))$ we take a representative $g_z \in \Gtilde^F$ and define the Gelfand--Graev character associated to the class $z$ by $\Gamma_z={}^{g_z} \Gamma_1.$

\section{A McKay-type bijection}\label{Maslowski}
We fix an indecomposable root system $\Phi$ of rank $n$. From now on $\G$ will denote a simple algebraic group of simply connected type with root system $\Phi$ defined over a field of characteristic $p$.

\subsection{Frobenius endomorphisms}\label{auto}

For $\alpha\in\Phi$ we fix isomorphisms $x_\alpha:(\mathbf{k},+)\to \U_\alpha$. If $f$ is a positive integer we consider the Frobenius endomorphism $F_{p^f}: \G \to \G$ given by $F_{p^f}(x_\alpha(a)):=x_{\alpha}(a^{p^f})$ for $\alpha \in \Phi$. Moreover, for every symmetry $\gamma$ of the Dynkin diagram associated to $\Delta$ there exists a graph automorphism $\gamma: \G \to \G$ given by $\gamma(x_\alpha(a)):=x_{\gamma (\alpha)}(a)$ for $a \in \mathbf{k}$ and $\alpha \in \pm \Delta$.

Up to inner automorphisms of $\G$, every Frobenius endomorphism $F$ of $\G$ defining an $\mathbb{F}_q$-structure is of the form $F_q \gamma$ for some symmetry $\gamma$ and we may thus assume that $F=F_q \gamma$. Let $w$ denote the order of $\gamma$. We say that $F$ is a standard Frobenius endomorphism if $F=F_q$.

\subsection{A regular embedding}\label{regular}

The center of $\G$ is a finite group. We let $d_p$ be its minimal number of generators. For our fixed root system $\Phi$ we let $d$ be the maximal $d_p$ occurring for any prime $p$. We have $d=2$ if the root system of $\G$ is of type $D_{n}$ and $n$ is even. In all other cases we either have $d=1$ or $d=0$.

Let $\mathbf{S}=(\mathbf{k}^{\times})^d$ be a torus of rank $d$. Let $\rho: \Z(\G)\hookrightarrow \mathbf{S}$ be an injective group homomorphism and define a group $\Gtilde$ by
\begin{align*}
\Gtilde=\G\times_\rho \mathbf{S}= (\G \times \mathbf{S}) / \{(z,\rho(z)^{-1}) \mid z\in \Z(\G)\}.
\end{align*}
We have embeddings $\G\hookrightarrow \Gtilde$ and $\mathbf{S}\hookrightarrow \Gtilde$ such that it is convenient to identify $\G$ and $\mathbf{S}$ with their images in $\Gtilde$. Under this identification $\Gtilde=\G \mathbf{S}$ has connected center $\Z(\Gtilde)=\mathbf{S}$. Note that the construction of $\Gtilde$ depends on the choice of $\rho$. In \cite[Section 6]{Maslowski2010} explicit choices are made which we assume to be taken.

Note that $\Btilde=\Ttilde \U=\N_{\Gtilde}(\U)$ is a Borel subgroup of $\Gtilde$ with maximal torus $\Ttilde=\T \mathbf{S}$ and unipotent radical $\U$.  We also denote by $F_p$ and $\gamma$ the extensions of the bijective morphisms from \ref{auto} to $\tilde{\G}$ as chosen in Section 3 and Section 4 of \cite{Maslowski2010}. 
%With theses choices, the pair $(\Ttilde,\Btilde)$ is $F$-stable.

\subsection{Generators of the torus}

Recall that $w$ denotes the order of the graph automorphism involved in $F$. Let us define an integer $\bar{d}$ by
$$\bar{d}=\begin{cases} 1 & \text{if $(\G,F)$ is of type $D_n$, $n$ even and $w=2$,} \\
0 & \text{if $(\G,F)$ is of type $D_n$, $n$ even and $w=3$,} \\
	d & \text{otherwise. }
\end{cases}$$
If $\bar{d}=0$ we define $t_0=1$. If $\bar{d}=1$ we let $t_0$ be the generator of $\Z(\Gtilde^F)$ as in \cite[Section 10]{Maslowski2010}. If $\bar{d}=2$ we let $t_{0^{(1)}},t_{0^{(2)}}$ be the generators of $\Z(\Gtilde^F)$ as in \cite[Section 10]{Maslowski2010}. In this case, we mean by $t_0$ both elements $t_{0^{(1)}}$ and $t_{0^{(2)}}$.

Recall from \ref{Gelfandsection} that the integer $r$ denotes the number of $\gamma$-orbits of $\Delta$. We let $t_1, \dots, t_r \in \Ttilde^F$ be as introduced in \cite[Proposition 8.1]{Maslowski2010} resp. \cite[Proposition 10.2]{Maslowski2010} which together with $t_0$ generate the torus $\Ttilde^F$.

\subsection{The linear characters of $\U^F$}\label{loc}

Let us from now on assume that $\G^F$ is not of type $B_2(2)$, $F_4(2)$ or $G_2(3)$. In this case, we have $[\U,\U]^F=[\U^F,\U^F]$ by \cite[Lemma 7]{Howlett}. By Lemma \ref{characterU} we obtain a bijection
\begin{align*}
\delta: \Irr (\U^F / [\U^F,\U^F] )\to \prod_{i=1}^r (\mathbb{F}_{q^{|A_i|}},+). \end{align*}
Let $S$ be a subset of $\{1,\dots,r \}$. We denote $S^c=\{0,1,\dots,r \}\setminus S$. Define the character $\phi_S$ of $\U^F / [\U^F,\U^F]$ to be $\phi_S=\delta^{-1}(c_1, \ldots ,c_r)$ with 
$$c_i=\begin{cases} 0 & \text{if } i \notin S, \\
	1 & \text{if } i \in S.
\end{cases}$$
For simplicity we identify $\phi_S\in \Irr(\U^F / [\U^F,\U^F])$ with its inflation to $\U^F$. Note that with this notation the linear character $\xi$ introduced in \ref{Gelfandsection} coincides with $\phi_{\{1,\ldots,r \}}$.

The action of $\Ttilde^F$ on the characters of $\U^F$ can be described explicitly and one obtains the following result.

\begin{proposition}\label{charactersU} The characters $\{ \phi_S\in\Irr(\U^F) \mid S\subseteq \{1,\dots,r \} \}$ form a complete set of representatives for the $\Btilde^F$-orbits on the linear characters of $\U^F$.  Moreover any character $\phi_S\in\Irr(\U^F)$ extends to its inertia group $\Btilde^F_{\phi_S}=\langle t_i \mid i \in S^c \rangle \U^F$.
\end{proposition}
\begin{proof}
	This is \cite[Proposition 8.4]{Maslowski2010} and \cite[Proposition 8.5]{Maslowski2010}.
\end{proof}

As a consequence of the previous proposition we can describe the action of Galois automorphisms on linear characters of $\U^F$.

\begin{lemma}\label{galoisU}
	Let $\sigma\in\mathcal{G}$. Then there exists some $\tilde{t} \in \Ttilde^F$ such that $\phi_S^\sigma=\phi_S^{\tilde{t}}$ for every $S \subseteq \{1, \ldots, r\}$.
\end{lemma}

\begin{proof}
Let $\mathcal{S}= \{1, \ldots, r \}$. By the uniqueness statement of Proposition \ref{charactersU} we have  $\phi_\mathcal{S}^\sigma=\phi_{\mathcal{S}^{'}}^{\tilde{t}}$ for some $\mathcal{S}^{'} \subseteq \{1,\dots ,r \}$  and some $\tilde{t} \in \Ttilde^F$. Recall that the subgroups $\U_{A_i}^F$ are stabilized by the $\Ttilde^F$-action. Thus, we have 
 
 \begin{align*} \prod_{i \in \mathcal{S}^{'}} \phi_i^{\tilde{t}}=\phi_{\mathcal{S}^{'}}^{\tilde{t}}=\phi_\mathcal{S}^\sigma=\prod_{i\in \mathcal{S}} \phi_i^\sigma. \end{align*}
 Since the characters $\phi_i,\phi_i^\sigma \in\operatorname{Irr}(\U_{A_i}^F)$ are nontrivial this implies $\mathcal{S}=\mathcal{S}^{'}$ and $\phi_\mathcal{S}^\sigma=\phi_\mathcal{S}^{\tilde{t}}$. Hence for every $S \subseteq \mathcal{S}$ we have $\phi_S^\sigma = \phi_S^{\tilde{t}}$ as well.
\end{proof}

\subsection{A labeling for the local characters}\label{loclabel}

We can now parametrize the $p'$-characters of $\Btilde^F$. Let $\psi\in\Irr_{p'}(\Btilde^F)$. Since $\U^F$ is a normal $p$-subgroup of $\Btilde^F$ and $\psi$ has $p'$-degree it follows by Clifford's theorem that $\psi$ lies above a linear character of $\U^F$. Hence, by Proposition \ref{charactersU} there exists a uniquely determined subset $S \subseteq \{1,\dots,r\}$ such that $\psi$ lies above $\phi_S\in\Irr(\U^F)$.  By Clifford correspondence there exists a unique character $\lambda\in\Irr(\Btilde^F_{\phi_S}\mid \phi_S)$ with $\lambda^{\Btilde^F}=\psi$. Note that $\Btilde^F_{\phi_S}=\langle t_i \mid i \in S^c \rangle \U^F$ by Proposition \ref{charactersU}. We define the map $\tilde{f}_{\mathrm{loc}}:\Irr_{p'}(\Btilde^F)\to (K^\times)^{\bar{d}} \times K^r$ by
$$
		(\tilde{f}_{\mathrm{loc}}(\psi))_0=\begin{cases}
		\lambda(t_0) & \text{if } \bar{d}\leq 1, \\
		(\lambda(t_{0^{(1)}}),\lambda(t_{0^{(2)}})) & \text{if } \bar{d}=2,
				\end{cases}$$
and
$$(\tilde{f}_{\mathrm{loc}}(\psi))_i=\begin{cases}
		\lambda(t_i) & \text{if } i\in S^c \setminus \{0 \}, \\
		0 & \text{if } i\in S,
				\end{cases}$$
where $\lambda$ is determined by $\psi$ as above. For $i\in S^c$ the values $\lambda(t_i)$ of the linear character $\lambda$ are $(q^w-1)$-th roots of unity, where $w$ is as defined in \ref{Gelfandsection}. The elements of $p'$-order of $K^\times$ are in the image of the embedding $\mathbf{k}^\times \to K^\times$ chosen in \ref{rings}.  Hence, we may consider $(\tilde{f}_{\mathrm{loc}}(\psi))_i$ as an element of $\mathbb{F}_{q^w}$ and we obtain a map $\tilde{f}_{\mathrm{loc}}:\Irr_{p'}(\Btilde^F)\to (\mathbb{F}_{q^w}^\times)^{\bar{d}} \times \mathbb{F}_{q^w}^r$.
Let $\mathcal{A}\subseteq (\mathbb{F}_{q^w}^\times)^{\bar{d}} \times \mathbb{F}_{q^w}^r$ be the image of the map $\tilde{f}_{\mathrm{loc}}$. By \cite[Theorem 10.8]{Maslowski2010} the map $\tilde{f}_{\mathrm{loc}}:\Irr_{p'}(\Btilde^F)\to \mathcal{A}$ is injective and hence a bijection.

\subsection{The dual group}\label{dualgroup}

We give an explicit construction of the dual algebraic group of $\Gtilde$, following the construction in \cite[Section 7]{Maslowski2010}. It is similar to the construction in \ref{regular}. Let $\mathbf{G}^\vee$ be a simple algebraic group of simply connected type with root system $\Phi^\vee$. We fix a maximal torus $\T^\vee$ of $\G^\vee$ and identify the root system of $\G^\vee$ relative to the torus $\T^\vee$ with the coroot system $\Phi^\vee$.

We let $\mathbf{S}^\vee=(\mathbf{k}^\times)^d$, where $d$ is as in \ref{regular}, and we choose an injective group homomorphism $\rho^\vee:\Z(\mathbf{G}^\vee)\to\mathbf{S}^\vee$ as in \cite[Section 7]{Maslowski2010}. Denote by $\Gtilde^{*}$ the resulting linear algebraic group $\Gtilde^*=\G^\vee \times_{\rho^\vee} \mathbf{S}^\vee$ with maximal torus $\Ttilde^*:=\T^\vee \mathbf{S}^\vee$. By the results of \cite[Section 7]{Maslowski2010} there exists a Frobenius endomorphism $F^*$ of $\Gtilde^*$ such that $(\Gtilde,\Ttilde,F)$ is dual to the triple $(\Gtilde^{*},\Ttilde^*,F^{*})$.

\subsection{Fundamental weights}

Since $\G^\vee$ is a simple algebraic group of simply connected type its character group $X(\T^\vee)$ has a basis given by the fundamental weights. More precisely, let $\beta_1,\dots,\beta_n\in X(\T^\vee)$ be a basis of the root system $\Phi^\vee$ (corresponding to $\alpha_1^\vee,\dots,\alpha_n^\vee$ under the identification of the root system of $\G^\vee$ with $\Phi^\vee$). Denote by $\langle \: , \: \rangle: X(\T^\vee) \times Y(\T^\vee) \to \mathbb{Z}$ the canonical pairing. Then there exist weights $\omega_i\in X(\T^\vee)$ satisfying $\langle \omega_i,\beta_j^\vee \rangle=\delta_{ij}$ for all $i,j=1,\dots,n$. Moreover, we let $\tilde{\omega_i}\in X(\Ttilde^*)$ be the unique extension of $\omega_i$ to $\Ttilde^*$ which acts trivially on $\mathbf{S}^\vee$.

\subsection{The determinant map}

Write $g\in \Gtilde^*$ as $g=xz$ with $x\in \G^\vee$ and $z\in \mathbf{S}^\vee$. We define the determinant map $\det:\Gtilde^*\to \mathbf{S}^\vee$ to be the map with $\det(xz)=z^{l}$ where $l$ is the exponent of the fundamental group of the root system $\Phi$.

Note that the map $\det$ is a well-defined homomorphism of algebraic groups, see the remark below \cite[Definition 7.2]{Maslowski2010}. Furthermore, we denote by $\det_i:\Gtilde^*\to \mathbf{k}^\times$ the $i$-th component of the determinant map.

\subsection{The modified Steinberg map}\label{modStein}

Following Maslowski \cite[Section 14]{Maslowski2010}, we introduce the modified Steinberg map which separates the semisimple conjugacy classes of $\Gtilde^*$.

By a theorem of Chevalley, see \cite[Theorem 15.17]{Malle2011}, there exists a rational irreducible $\mathbf{k}\G^\vee$-module $V_i$ which is a highest weight module of highest weight $\omega_i\in Y(\T^\vee)$. Let $\pi_i:\G^\vee \to \mathbf{k}$ denote the trace function of the representation associated to the $\mathbf{k} \G^\vee$-module $V_i$. We define the Steinberg map 
$$\pi:\G^\vee\to \mathbf{k}^n:g\mapsto(\pi_1(g),\dots,\pi_n(g))$$
as the product map of these trace functions.
A fundamental property of the Steinberg map is that two semisimple elements of $\G^\vee$ are $\G^\vee$-conjugate if and only if they have the same image under the Steinberg map, see \cite[Corollary 6.7]{Steinberg}.

%We need a further property of the Steinberg map which is a slight generalization of \cite[Lemma 14.1]{Maslowski2010} and follows from the proof given there.
%
%\begin{lemma}\label{Steinberg}
%Let $s\in\G^\vee$ be a semisimple element. Then $\pi_i(s^p)=\pi_i(s)^p$.
%\end{lemma}

We can write any element $\tilde{g}\in\Gtilde^*$ (not necessarily unique) as $\tilde{g}=xz$ with $x\in \G^\vee$ and $z\in\mathbf{S}^\vee$. In \cite[Section 14]{Maslowski2010} Maslowski defines the map $\tilde{\pi}:\Gtilde^* \to (\mathbf{k}^\times)^{d} \times \mathbf{k}^n$ by 
\begin{align*}
\tilde{g}=xz\mapsto (\operatorname{det}(xz),(\pi_1(x) \tilde{\omega}_1(z),\dots,\pi_n(x) \tilde{\omega}_n(z))). 
\end{align*}
Based on the result of Steinberg mentioned above, Maslowski shows in \cite[Proposition 14.2]{Maslowski2010} that the map $\tilde{\pi}$ separates semisimple conjugacy classes of $\Gtilde^*$. Moreover, the semisimple $\tilde{\G}^\ast$-conjugacy classes of elements $s \in \tilde{\G}^\ast$ with image $\tilde{\pi}(s)$ in $(\mathbb{F}_q^\times)^d \times \mathbb{F}_q^n$ are precisely the $(q-1)^d q^n$ different $F_q^\ast$-stable semisimple conjugacy classes of $\Gtilde^*$. 

\subsection{A labeling for the global characters}\label{globallabel}

%We denote by $[\tilde{s}]$ the $\Gtilde^*$-conjugacy class of some semisimple element $\tilde{s}\in \Gtilde^*$. If $[\tilde{s}]$ is $F^*$-stable it contains a rational element by the Theorem of Lang--Steinberg. Thus, in this case we may assume $\tilde{s}\in \Gtilde^*$. Moreover, note that since the center of $\Gtilde$ is connected it follows that $F^*$-stable semisimple conjugacy classes of $\Gtilde^*$ are in one to one correspondence with semisimple conjugacy classes of $\Gtilde^*$ (see \cite[3.25]{Digne1991} and \cite[Remark 13.15 (ii)]{Digne1991}

We now describe a labeling for the $p'$-characters of $\Gtilde^F$. Let $\chi\in\Irr_{p'}(\Gtilde^F)$ be a $p'$-character. Then there exists a conjugacy class $(\tilde{s})$ of $(\Gtilde^\ast)^{F^\ast}$ such that $\chi\in \mathcal{E}(\Gtilde^F,(\tilde{s}))$.

We first consider the case that $F=F_q$. In this case, we define the label of $\chi$ by $\tilde{\pi}(\tilde{s})=(b_0,(b_1,\dots,b_n))\in (\mathbb{F}_q^\times)^d \times \mathbb{F}_q^n$.
	
Now suppose that $F$ is not a standard Frobenius map. Since $(\Gtilde^\ast)^{F^\ast} \subseteq (\Gtilde^\ast)^{F_{q^w}^\ast}$ we have $\tilde{s} \in  (\Gtilde^\ast)^{F_{q^w}^\ast}$. In particular it holds $\tilde{\pi}(\tilde{s})=(b_0,(b_1,\dots,b_n))\in(\mathbb{F}_{q^w}^\times)^d \times \mathbb{F}_{q^w}^n$. We define the label of $\chi$ by $(b_{0^{(1)}},(b_{a_1},\dots,b_{a_r}))\in (\mathbb{F}_q^\times)^{\bar{d}} \times \mathbb{F}_q^r$, where $a_i\in A_i$ are the fixed representatives of the orbits of the $\gamma$-action and $b_{0^{(1)}}$ is the first component of $b_0\in (\mathbb{F}_{q^w}^\times)^d$. 

In any case, the possible labels which occur consist precisely of the elements of $\mathcal{A}$, where $\mathcal{A}$ is defined as in \ref{loc}. We shall denote by $\tilde{f}_{\mathrm{glo}}:\Irr_{p'}(\Gtilde^F) \to \mathcal{A}$ the map which sends a character to its label.
 
\subsection{The Maslowski bijection and its properties}\label{subbijection} 
From now on we often write $H=\mathbf{H}^F$ for the group of fixed points under $F$ of an $F$-stable subgroup $\mathbf{H}$ of $\Gtilde$. In most cases, the map $\tilde{f}_{\mathrm{glo}}:\Irr_{p'}(\Gtilde^F) \to \mathcal{A}$ is known to be bijective.
\begin{theorem}\label{bijection}
Suppose that $(\G,F)$ is not contained in the following table.
\\
\\
\begin{tabularx}{8cm}{p{0.25\textwidth}|X|X|}
type & Frobenius map \\
%\hline
%$D_n$ & $q=2,w=2$ \\
\hline
$B_n$, $C_n$, $D_n$, $G_2$, $F_4$ & $q=2$, $w=1$ \\
\hline
$G_2$ & $q=3,w=1$ \\
\end{tabularx}
\\
\\
\\
Then the map $\tilde{f}_{\mathrm{glo}}:\Irr_{p'}(\tilde{G}) \to \mathcal{A}$ is a bijection. Consequently the map $\tilde{f}=\tilde{f}_{\mathrm{glo}}^{-1} \circ \tilde{f}_{\mathrm{loc}}:\Irr_{p'}(\tilde{B})\to \Irr_{p'}(\tilde{G})$ is a bijection. Moreover, for every character $\lambda \in \Irr(\Z(\tilde{G}))$ the bijection $\tilde{f}$ restricts to a bijection $\Irr_{p'}(\tilde{B} \mid \lambda)\to \Irr_{p'}(\tilde{G} \mid \lambda)$.
\end{theorem}

\begin{proof}
See \cite[Theorem 15.3]{Maslowski2010} and the remarks following \cite[Proposition 3.4]{Spaeth}.
\end{proof}

We shall keep the assumptions of Theorem \ref{bijection} for the remainder of this article.

\section{The McKay Conjecture and Galois automorphisms}\label{section5}

This section is roughly divided in three parts. In \ref{compatibility} we show that the bijection $\tilde{f}:\operatorname{Irr}_{p'}(\Btilde^F)\to\operatorname{Irr}_{p^{'}}(\Gtilde^F)$ is $\mathcal{H}$-equivariant. After this, we relate the $p'$-characters of $\Btilde^F$ (resp. of $\Gtilde^F$) with the $p'$-characters of $\B^F$ (resp. of $\G^F$). In \ref{mainsection} we use these results to provide a proof of Theorem \ref{maintheorem} from the introduction.

 We keep the assumptions of Theorem \ref{bijection}.

\subsection{Compatibility of the character bijection with Galois automorphisms}\label{compatibility}

%The following properties of $(e,p)$-Galois automorphisms are elementary and easy to prove.

%\begin{lemma}\label{epGalois}
%Let $\sigma\in \Gal(\mathbb{Q}_m / \mathbb{Q})$ and $k$ be an integer such that $\sigma(\xi)=\xi^k$ for a primitive $m$-th root of unity $\xi\in \mathbb{Q}_m$.
%	\begin{compactenum}[(a)] 
%		\item Then $\sigma$ is an $(e,p)$-Galois automorphism if and only if $k \equiv p^e \mod m_{p'}$.
%		\item Let $\tilde{m}$ be a multiple of $m$. Then any $(e,p)$-Galois automorphism $\sigma\in \Gal(\mathbb{Q}_m/\mathbb{Q})$ extends to an $(e,p)$-Galois automorphism $\tilde{\sigma}\in\Gal(\mathbb{Q}_{\tilde{m}}/\mathbb{Q})$.
%	\end{compactenum}
%\end{lemma}
%
%Note that part (b) of Lemma \ref{epGalois} implies that any $(e,p)$-Galois automorphism $\sigma\in \Gal(\mathbb{Q}_{|G|}/\mathbb{Q})$ extends to an $(e,p)$-Galois automorphism $\tilde{\sigma}\in \Gal(\mathbb{Q}_{|\tilde{G}|}/\mathbb{Q})$. This means that if we want to prove Conjecture \ref{Navarro} for the finite group $G$ and an $(e,p)$-Galois automorphism $\sigma\in \Gal(\mathbb{Q}_{|G|}/\mathbb{Q})$, we may (and we will) assume without loss of generality that $\sigma\in \Gal(\mathbb{Q}_{|\tilde{G}|}/\mathbb{Q})$.

We now show that the bijection $\tilde{f}:\Irr_{p'}(\tilde{B}) \to \Irr_{p'}(\tilde{G})$ is $\mathcal{H}$-equivariant. In the following proof we freely use the notation introduced in Section \ref{Maslowski}.

	\begin{theorem}\label{sigma} 
The bijection
$
		\tilde{f}:\operatorname{Irr}_{p'}(\tilde{B})\to\operatorname{Irr}_{p'}(\tilde{G})
$
is $\mathcal{H}$-equivariant.

%Moreover, if we denote the label of $\psi$ by $\tilde{f}_{\mathrm{loc}}(\psi)=(c_0,(c_1\dots,c_r))\in (\mathbb{F}_{q^w}^\times)^{\bar{d}} \times \mathbb{F}_{q^w}^r$ then the label of $\psi^\sigma$ is given by $\tilde{f}_{\mathrm{loc}}(\psi^\sigma)=(c_0^{p^e},(c_1^{p^e},\dots,c_r^{p^e}))$. 
\end{theorem}

	\begin{proof}
Let $\psi\in \operatorname{Irr}_{p'}(\tilde{B})$ with label $\tilde{f}_{\mathrm{loc}}(\psi)=(c_0,(c_1,\ldots,c_r))$. Denote by $\chi=\tilde{f}(\psi)$ the $p'$-character of $\tilde{G}$ which has the same label as $\psi$. Fix a Galois automorphism $\sigma \in \mathcal{H}$ that sends any $p'$-root of unity $\zeta$ to $\zeta^{p^e}$.

We proceed in several steps. In a first step we compute the label $g(\psi^\sigma)$ of the character $\psi^\sigma$. In a second step we prove that $\tilde{f}_{\mathrm{glo}}(\chi^\sigma)=\tilde{f}_{\mathrm{loc}}(\psi^\sigma)$ which implies $\tilde{f}(\psi^\sigma)=\chi^\sigma$ since $\tilde{f}=\tilde{f}_{\mathrm{glo}}^{-1} \circ \tilde{f}_{\mathrm{loc}}$.
	\\
	\textbf{First step:} By Proposition \ref{charactersU} there exists a unique set $S\subseteq \{1,\dots,r \}$ such that the character $\psi\in\Irr_{p'}(\tilde{B})$ lies above the character $\phi_S\in\operatorname{Irr}(U)$. By Clifford correspondence there exists a unique character $\lambda\in \operatorname{Irr}(\tilde{B}_{\phi_S}\mid \phi_S)$ such that $\lambda^{\tilde{B}}=\psi$. 
	
Since $\psi$ lies above $\phi_S$ it follows that $\psi^\sigma$ lies above the character $\phi_S^\sigma$. By Lemma \ref{galoisU} we have $\phi_S^\sigma=\phi_S^{\tilde{t}}$ for some $\tilde{t}\in\tilde{T}$. The character $\lambda^\sigma$ lies above the character $\phi_S^\sigma=\phi_S^{\tilde{t}}$. Since the factor group $\tilde{B}/U \cong \tilde{T}$ is abelian we have $(\lambda^\sigma)^{\tilde{t}^{-1}}\in \Irr(\tilde{B}_{\phi_S})$. Consequently, $(\lambda^\sigma)^ {\tilde{t}^{-1}}$ lies above the character $\phi_S$ and $((\lambda^\sigma)^ {\tilde{t}^{-1}})^ {\tilde{B}}=\psi^\sigma$. Let $m=|\tilde{G}|$ and $\zeta$ be a primitive $m$-th root of unity. We write $m_p$ for the highest $p$-power dividing $m$ and $m_{p'}$ for the $p'$-part of $m$. Furthermore we let $k$ be an integer such that $\sigma(\zeta)=\zeta^k$. Since $\lambda$ is linear we have $\lambda^\sigma=\lambda^k$. Thus we obtain
$$(\tilde{f}_{\mathrm{loc}}(\psi^\sigma))_i=(\lambda^k)^{\tilde{t}^{-1}}(t_i)=(\lambda^k)(\tilde{t} t_i \tilde{t}^{-1})=\lambda^k(t_i).$$
By definition of the map $\tilde{f}_{\mathrm{loc}}$ we have $(\tilde{f}_{\mathrm{loc}}(\psi^\sigma))_i=c_i^k$ for every $i\in S^c$ and $(\tilde{f}_{\mathrm{loc}}(\psi^\sigma))_i=0$ for $i\in S$. Consequently, the label of $\psi^\sigma$ is given by $\tilde{f}_{\mathrm{loc}}(\psi^\sigma)=(c_0^k,c_1^k,\dots,c_r^k)$.

Since $ \sigma \in \mathcal{H}$ we have $k \equiv p^e \mod m_{p'}$. By \cite[Table 24.1]{Malle2011} it follows that $q^w-1$ divides $m_{p'}$, which implies that $k \equiv p^e \mod(q^w-1)$. Since $c_i\in \mathbb{F}_{q^w}$ for all $i$, we have
 $$ \tilde{f}_{\mathrm{loc}}(\psi^\sigma)=(c_0^k,(c_1^k,\dots,c_r^k))=(c_0^{p^e},(c_1^{p^e},\dots,c_r^{p^e})).$$
\\
\textbf{Second step:} We have $\chi\in\mathcal{E}(\Gtilde^F,(\tilde{s}))$ for some semisimple conjugacy class $(\tilde{s})$ of the dual group $(\Gtilde^\ast)^{F^\ast}$. By Lemma \ref{sigmadl} we have $\chi^{\sigma}\in\mathcal{E}(\Gtilde^F,(\tilde{s}^k))$. We have $m=|(\Gtilde^\ast)^{F^\ast}|$ since $(\Gtilde,F)$ and $(\Gtilde^*,F^*)$ are in duality, see \cite[Proposition 4.4.4]{Carter1985}. Thus, the order of the semisimple element $\tilde{s}$ is a divisor of $m_{p'}$. Since $k\equiv p^e \mod m_{p'}$ this shows $\tilde{s}^ {p^e}=\tilde{s}^k$. Hence, we have $\chi^\sigma \in \mathcal{E}(\Gtilde^F,(\tilde{s}^{p^e}))$.
	
First we assume that $F=F_q$ is a standard Frobenius map. We may write $\tilde{s}\in \Gtilde^*$ as $\tilde{s}=xz$ where $x\in\G^\vee$ and $z\in\mathbf{S}^\vee$. The label of the character $\chi^\sigma\in \mathcal{E}(\Gtilde^F,(\tilde{s}^{p^e}))$ is given by
	\begin{align*}
	\tilde{\pi}(\tilde{s}^{p^e})=\tilde{\pi}((xz)^{p^e})=(\operatorname{det}((xz)^{p^e}),(\pi_1(x^{p^e}) \tilde{\omega}_1(z^{p^e}),\dots,\pi_n(x^{p^e})\tilde{\omega}_n(z^{p^e}))).
		\end{align*}
Note that $\operatorname{det}((xz)^{p^e})=\det(xz)^{p^e}$ since $\det$ is multiplicative. Recall that $\tilde{\pi}_i(\tilde{s}^{p^e})=\pi_i(x^{p^e}) \tilde{\omega}_i(z^{p^e})$ by definition of the modified Steinberg map. By \cite[Lemma 14.1]{Maslowski2010} we have $\pi_i(x^{p^e})=\pi_i(x)^{p^e}$. Moreover, we have $\tilde{\omega}_i(z^{p^e})=\tilde{\omega}_i(z)^{p^ e}$ since $\tilde{\omega}_i\in X(\Ttilde^*)$. Therefore we obtain $\tilde{\pi}_i(\tilde{s}^{p^e})=\pi_i(x)^{p^e} \tilde{\omega}_i(z)^{p^e}=\tilde{\pi}_i(\tilde{s})^{p^e}$. Since $\tilde{\pi}(xz)=(c_0,(c_1,\dots,c_n))$ we have $\tilde{\pi}((xz)^{p^e})=(c_0^{p^e},(c_1^{p^e},\dots,c_n^{p^e}))$ and therefore the label of $\chi^\sigma$ is given by $\tilde{f}_{\mathrm{glo}}(\chi^\sigma)=(c_0^{p^e},(c_1^{p^e},\dots,c_n^{p^e}))=\tilde{f}_{\mathrm{loc}}(\psi^\sigma)$ and we have $\tilde{f}(\psi^\sigma)=\tilde{f}(\psi)^\sigma$, as required.

Let us now assume that $F$ is not a standard Frobenius endomorphism. Let $\tilde{\pi}(\tilde{s})=(b_0,(b_1,\dots,b_n))\in (\mathbb{F}_{q^w}^\times)^{d} \times \mathbb{F}_{q^w}^n$ be the image of $\tilde{s} \in (\Gtilde^\ast)^{F^\ast} \subseteq (\Gtilde^\ast)^{F_{q^w}^\ast}$ under the modified Steinberg map. As we have shown above, the image of $(\tilde{s}^{p^e})$ under the modified Steinberg map is given by $\tilde{\pi}(\tilde{s}^{p^e})=(b_0^{p^e},(b_1^{p^e},\dots,b_n^{p^e}))$. Let $b_{0^{(1)}}$ be the first component of $b_0 \in (\mathbb{F}_{q^w}^\times)^{\bar{d}}$. The label of the character $\chi\in \mathcal{E}(\Gtilde^F,(\tilde{s}))$ is given by $\tilde{f}_{\mathrm{glo}}(\chi)=(b_{0^{(1)}},(b_{a_1},\dots,b_{a_r}))$ and the label of $\chi^\sigma\in\mathcal{E}(\Gtilde^F,(\tilde{s}^{p^e}))$ is $\tilde{f}_{\mathrm{glo}}(\chi^\sigma)=(b_{0^{(1)}}^{p^e},(b_{a_1}^{p^e},\dots,b_{a_r}^{p^e}))$. We have $\tilde{f}_{\mathrm{loc}}(\psi)=(c_0,(c_1,\dots,c_r))=(b_{0^{(1)}},(b_{a_1},\dots,b_{a_r}))=\tilde{f}_{\mathrm{glo}}(\chi)$ since $\tilde{f}(\psi)=\chi$. Thus, we conclude that  $$\tilde{f}_{\mathrm{loc}}(\psi^\sigma)=(c_0^{p^e},(c_1^{p^ e},\dots,c_r^{p^e}))=(b_{0^{(1)}}^{p^e},(b_{a_1}^{p^ e},\dots,b_{a_r}^{p^e}))=\tilde{f}_{\mathrm{glo}}(\chi^\sigma).$$
This shows $\tilde{f}(\psi^\sigma)=\chi^\sigma$, as desired.
\end{proof}

The following remark will be crucial in the upcoming calculations.

\begin{remark}\label{field}
	Let $\sigma \in \mathcal{H}$ be a Galois automorphism such that $\sigma$ sends any $p'$-root of unity $\zeta$ to $\zeta^{p^e}$. Let $\chi \in \Irr_{p'}( \tilde{G})$ and $\psi \in \Irr_{p'}( \tilde{B})$. By the proof of Theorem \ref{sigma} and \cite[Proposition 14.1]{Maslowski2010} the characters $\chi^\sigma$ and $\chi^{F_{p^e}}$ have the same label, which implies that $\chi^\sigma= \chi^{F_{p^e}}$. The same argument (using \cite[Proposition 9.4]{Maslowski2010} instead of \cite[Proposition 14.1]{Maslowski2010}) shows that $\psi^\sigma= \psi^{F_{p^e}}$.
\end{remark}

Theorem \ref{sigma} gives us Theorem \ref{maintheorem} in the case where $\Z(G)$ is trivial.

\begin{corollary}\label{trivialcenter}
 Suppose that $\Z(G)=1$. Then there exists an $\mathcal{H}$-equivariant bijection $f:\Irr_{p'}(B) \to \Irr_{p'}(G)$.
\end{corollary}

\begin{proof}
By Theorem \ref{bijection} and Theorem \ref{sigma} we obtain an $\mathcal{H}$-equivariant bijection $\Irr_{p'}(\tilde{B} \mid \lambda) \to \Irr_{p'}(\tilde{G} \mid \lambda)$ for any $\lambda \in \Z(\tilde{G})$.
Since $\Z(G)=1$ we have $\tilde{G} \cong  G \times \Z(\tilde{G})$. By Theorem \ref{sigma} we have an $\mathcal{H}$-equivariant bijection $\tilde{f}:\Irr_{p'}(\tilde{B} \mid \lambda^{\mathcal{H}}) \to \Irr_{p'}(\tilde{G} \mid \lambda^{\mathcal{H}})$ for every central character $\lambda \in \Z(\tilde{G})$. So in particular, for $\lambda=1_{\Z(\tilde{G})}$ we obtain a bijection $f:\Irr_{p'}(B) \to \Irr_{p'}(G)$.
\end{proof}

\subsection{Group automorphisms}

We denote by $D$ the subgroup of $\mathrm{Aut}(\tilde{\G}^F)$ generated by the restrictions to $\Gtilde^F$ of the graph automorphisms $\gamma: \Gtilde \to \Gtilde$ which commute with $F$ and the Frobenius endomorphism $F_p: \tilde{\G}^F \to \tilde{\G}^F$ as in \ref{auto} and \ref{regular}. Note that $\B^F$ is $D$-invariant. We may and we will choose the character $\phi_0 \in \Irr( \mathbb{F}_{q^w}, +)$ in \ref{Gelfandsection} such that it has order $p$. A consequence of that choice is that the characters $\phi_S$, for $S \subseteq \{1, \dots, r \}$ are $F_p$-stable. In particular, the character $\xi= \phi_{ \{1, \ldots, r \} }$ is $D$-stable and so is the Gelfand--Graev character $\Gamma_1=\xi^{\G^F}$.  

\subsection{Describing the $p'$-characters of $\G^F$}\label{pprime}

From now we work with the following assumption, see \cite[Assumption 3.2]{Spaeth}:

\begin{assumption}\label{spaethassumption}
The group $G=\G^F$ satisfies 
	\[G\notin \{ B_n(2), G_2(2), B_2(2^i), G_2(3^i), F _4(2^i) \mid n \geq 2, i \geq 1 \}.\] 
\end{assumption}

%In particular, the assumptions of Theorem \ref{bijection} are satisfied.

We denote by $\D_\G:\mathbb{Z}\Irr(\G^F)\to \mathbb{Z}\Irr(\G^F)$ the Alvis--Curtis duality map, see \cite[Chapter 8]{Digne1991}.
%Note that $\D_\G$ is an involution and an isometry on the space $\mathbb{Z}\Irr(\G^F)$, see \cite[Corollary 8.14]{Digne1991} resp. \cite[Proposition 8.10]{Digne1991}.

%The following lemma describes the action of Galois automorphisms on the Alvis--Curtis dual of the Gelfand--Graev characters.
%
%\begin{lemma}\label{duality}
%	Let $\sigma\in\mathcal{G}$ and $\tilde{t} \in \tilde{T}$ such that $\xi^\sigma= \xi^{\tilde{t}}$. Then $\D_\G(\Gamma_1)^\sigma= \D_\G(\Gamma_1)^{\tilde{t}}$.
%\end{lemma}
%
%\begin{proof}
%	Recall that the functor $\D_\G$ is defined using Harish--Chandra induction and restriction which is compatible with the action of Galois automorphisms. On the other hand the Levi subgroups and parabolic subgroups which occur in the definition of $\D_\G$ can be chosen such that they are stabilized by the $\tilde{T}$-conjugation. So we deduce that $\D_\G(\Gamma_1^\sigma)=\D_\G(\Gamma_1)^\sigma$ and $\D_\G(\Gamma_1^{\tilde{t}})=\D_\G(\Gamma_1)^{\tilde{t}}$. Since $\xi^\sigma=\xi^{\tilde{t}}$ we conclude $\Gamma_1^\sigma=\Gamma_1^{\tilde{t}}$. Consequently, we have $\D_\G(\Gamma_1)^\sigma=\D_\G(\Gamma_1)^{\tilde{t}}$.
%\end{proof}
%
%\subsection{Stabilizers of characters}\label{localglobal}

In what follows, we fix a Galois automorphism $\sigma \in \mathcal{H}$ such that $\sigma$ maps every $p'$-root of unity $\zeta \in K^\times$ to $\zeta^{p^e}$. By Lemma \ref{galoisU} there exists some $\tilde{t} \in \tilde{T}$ such that $\phi_S^\sigma= \phi_S^{\tilde{t}}$ for every $S \subseteq \{1, \ldots, r \}$.

\begin{lemma}\label{global}
For every $\chi \in \Irr_{p'}(\tilde{G})$ there exists a character 
$\chi_0 \in \Irr(G \mid \chi)$ which satisfies
	$$(\tilde{G} D \langle \sigma \rangle )_{\chi_0}=\tilde{G}_{\chi_0} D_{\chi_0} \langle F_{p^e} \sigma^{-1} \tilde{t} \rangle$$.
\end{lemma}

\begin{proof}
	According to the proof of \cite[Remark 3.4]{Spaeth} for every $\chi \in \Irr_{p'}(\tilde{G})$ there exists a unique character $\chi_0 \in \Irr(G \mid \chi)$ with $(\chi_0,\D_{\G}(\Gamma_1)) = \pm 1$. By Remark \ref{field} we have $\chi^{F_{p^e}}=\chi^{\sigma}$. Since $\Gamma_1$ is $D$-stable and $\sigma^{-1} \tilde{t}$ fixes $\Gamma_1$ it follows that $F_{p^e} \sigma^{-1} \tilde{t}$ fixes $\Gamma_1$. The duality functor $\D_\G$ commutes with Galois automorphisms and group automorphisms of $\G^F$. Hence, it follows that $F_{p^e} \sigma^{-1} \tilde{t}$ fixes $\D_\G(\Gamma_1)$. Moreover, as $\chi^{F_{p^e}}=\chi^{\sigma}$ it follows that $\chi_0^{F_{p^e} \sigma^{-1} \tilde{t}}$ is below $\chi$. Thus, $\chi_0^{F_{p^e} \sigma^{-1} \tilde{t}}$ is the unique common constituent of $\chi_G$ and $\D_{\G}(\Gamma_1)$. Therefore, $\chi_0^{F_{p^e} \sigma^{-1} \tilde{t}}= \chi_0$. Hence, $\tilde{G}_{\chi_0} D_{\chi_0} \langle F_{p^e} \sigma^{-1} \tilde{t} \rangle$ is contained in $(\tilde{G} D \langle \sigma \rangle )_{\chi_0}$.

	Now suppose that $\chi_0^{\tilde{g} d \sigma^{-1} }=\chi_0$. This rewrites as $\chi_0^{\tilde{g} d \tilde{t}^{-1} F_{p^e}^{-1}}= \chi_0$. By the proof of \cite[Remark 3.4]{Spaeth} we have $(\tilde{G} D)_{\chi_0}= \tilde{G}_{\chi_0} D_{\chi_0}$.
%	Since $\Gamma$ is $D$-stable it follows that $\tilde{t} d(\tilde{t})^{-1} \in \Z(\Gtilde^F) \G^F$. So we can rewrite this as $\chi_0^{F_{p^e}  d \tilde{t} \tilde{g}}= \chi_0$.
It follows that $\tilde{g} d(\tilde{t}^{-1}) \in  \tilde{G}_{\chi_0}$ and $d F_{p^e}^{-1} \in D_{\chi_0}$. Therefore, $\tilde{g} d \sigma^{-1}  \in (\tilde{G}_{\psi} D_\psi) \langle F_{p^e} \sigma^{-1} \tilde{t} \rangle$ which completes the proof.
\end{proof}

\begin{lemma}\label{local}
For every $\psi\in\Irr_{p'}(\tilde{B})$ there exists a character $\psi_0\in\Irr(B\mid \psi)$ such that
$$(\tilde{B} D \langle \sigma \rangle )_{\psi_0}=\tilde{B}_{\psi_0} D_{\psi_0} \langle F_{p^e} \sigma^{-1} \tilde{t} \rangle.$$
\end{lemma}

\begin{proof}
	Let $\psi \in \Irr(\tilde{B} \mid \phi_S)$. By Clifford correspondence there exists a unique character $\lambda \in \Irr( \tilde{B}_{\phi_S} \mid \phi_S)$ such that $\lambda^{\tilde{B}}= \psi$. Denote $I:=B_{\phi_S}$ and define $\psi_0:= \lambda_I^B$.
%	
%%	
%%	By \cite[Remark 3.6]{Spaeth} there exists $\psi_0\in\Irr(B\mid \psi)$ (WHICH ONE IS IT??? HOW ARE THE PARAMETRIZED AGAIN!!) such that
%	$$(\tilde{G} D)_{\psi_0}=\tilde{G}_{\psi_0} D_{\psi_0}.$$
	By Remark \ref{field} we have $\psi^{F_{p^e}}=\psi^{\sigma}$. Since $\phi_S$ is $F_p$-stable and $\sigma^{-1} \tilde{t}$ fixes $\phi_S$ it follows that $F_{p^e} \sigma^{-1} \tilde{t}$ fixes $\phi_S$. Consequently, $\lambda^{F_{p^e} \sigma^{-1} \tilde{t}} \in \Irr( \tilde{B}_{\phi_S} \mid \phi_S)$ is the Clifford correspondent of $\psi$. From this it follows that $\lambda^{F_{p^e} \sigma^{-1} \tilde{t}}= \lambda$. This implies that the character $\psi_0=\lambda_I^B$ is $F_{p^e} \sigma^{-1} \tilde{t}$-stable as well. Hence, the right-hand side is a subset of the left-hand side. By the proof of \cite[Remark 3.6]{Spaeth} the character $\psi_0$ satisfies $(\tilde{B} D)_{\psi_0}=\tilde{B}_{\psi_0} D_{\psi_0}.$ The converse can now be proved verbatim as in Lemma \ref{global}.
%		PLUS PROVE THE CONVERSE DIRECTION.
%	
%		Now suppose that $\psi_0^{\sigma d \tilde{g} }=\psi_0$. This is equivalent to $\psi_0^{F_{p^e} \tilde{t} d \tilde{g}}= \psi_0$. From this it follows that $\phi_S$ is $d$-stable.
%%		Hence, we obtain $\tilde{t} d(\tilde{t})^{-1} \in \Z(\Gtilde^F) \G^F$. So we can rewrite this as $\chi_0^{F_{p^e}  d \tilde{t} \tilde{g}}= \psi$.
%		
%	It follows that $F_{p^e} d d^{-1}(\tilde{t}) \tilde{g} \in  D_{\chi_0}$. Since  $(\tilde{B} D)_{\psi_0}=\tilde{B}_{\psi_0} D_{\psi_0}$ we conclude that $F_{p^e}d$ and $d^{-1}(\tilde{t}) \tilde{g}$ both stabilize $\psi_0$.
%	
\end{proof}

\subsection{An equivariant bijection for the Galois--McKay conjecture}\label{mainsection}

In the proof of the following theorem we closely follow the proof of \cite[Theorem 2.10]{Spaeth}. We set $\mathcal{B}:=\tilde{B}^F \rtimes D$.

\begin{theorem}\label{main}
	There exists an $\mathcal{H} \times \mathcal{B}$-equivariant bijection
	%	Let $\sigma\in\operatorname{Gal}(\mathbb{Q}_{G^F|}/\mathbb{Q})$ be an $(e,p)$-Galois automorphism. Suppose that $p$ is a good prime for $\G$. Then the bijection constructed in Theorem \ref{Mc} restricts to a bijection
	\begin{align*}
	f:\Irr_{p'}(B)\to \Irr_{p'}(G).
	\end{align*}
	%Moreover, for every central character $\lambda\in\Irr(\Z(G))$ the map $f$ restricts to a bijection $\Irr_{p'}(B\mid \lambda)^ \sigma\to \Irr_{p'}(G \mid \lambda)^\sigma$.
	
	\begin{proof}
		The groups $\mathcal{H} \times \mathcal{B}$ and $\Irr(\tilde{B}/B)$ both act on $\Irr_{p'}(\tilde{B})$. We let $\mathcal{T}$ be a transversal in $\Irr_{p'}(\tilde{B})$ with respect to these combined actions. For every $\psi \in \mathcal{T}$ we fix a character $\psi_0 \in \Irr( B \mid \psi)$ with the properties from Lemma \ref{local} and let $\mathcal{T}_0 \subseteq \Irr_{p'}(B)$ be the set formed by these. This is an $\mathcal{H} \times \mathcal{B}$--transversal in $\Irr_{p'}(B)$.
		
		The bijection $\tilde{f}$ is $(\mathcal{H} \times \mathcal{B}) \ltimes \Irr(\tilde{B}/B)$-equivariant by \cite[Theorem 3.5]{Spaeth} and Theorem \ref{sigma}.
		It follows that the set $\tilde{f}(\mathcal{T})$ is a transversal in $\Irr(\tilde{G}^F)$ with respect to the $(\mathcal{H} \times \mathcal{B}) \ltimes \Irr(\tilde{B}/B)$-action. For every $\chi \in \tilde{f}(\mathcal{T})$ we fix a character $\chi_0\in \Irr(G \mid \chi)$ satisfying the properties of Lemma \ref{global}. Denote by $\mathcal{T}_0'$ the set formed by these characters.
		
		For $\psi\in \mathcal{T}$ and $\psi_0\in \mathcal{T}_0$ we define $f(\psi_0):=\chi_0$, where $\chi_0$ is the unique element in $\mathcal{T}_0' \cap\Irr(G \mid \tilde{f}(\chi))$. By Lemma \ref{local} and Lemma \ref{global} we have $ (\mathcal{H} \times \mathcal{B})_{\psi_0}=(\mathcal{H} \times \mathcal{B})_{\chi_0}$ for every $\psi_0 \in \mathcal{T}_0$ and $\chi_0=f(\psi_0)$. We can therefore extend $f$ to an $\mathcal{H} \times \mathcal{B}$--equivariant bijection $f:\Irr_{p'}(B)\to \Irr_{p'}(G)$ 
		by setting $$ f(\psi_0^x):= f(\psi_0)^x\text{ for every } x \in \mathcal{H} \times \mathcal{B} \text{ and } \psi_0 \in \mathcal{T}_0. \qedhere $$
	\end{proof}

\end{theorem}

%\begin{proof}
%Let $f:\Irr_{p'}(B)\to \Irr_{p'}(G)$ be the bijection constructed in the proof of \ref{Mc}. Thus, it is sufficient to see that $\vartheta$ is $\sigma$-invariant if and only if $f(\vartheta)$ is $\sigma$-invariant for $\vartheta\in\Irr_{p'}(B)$. Let $\psi\in\Irr_{p'}(\tilde{B} \mid \vartheta)$. By construction of $f$ it follows that $f(\vartheta) \in \Irr(G \mid \tilde{f}(\psi))$. By Corollary \ref{invariantG} it follows that the character $\vartheta$ is $\sigma$-invariant if and only if $f(\vartheta)$ is $\sigma$-invariant. The statement about compatibility with central characters follows from the statement about central characters in Theorem \ref{Mc}.
%\end{proof}
\ \\
\noindent
\textit{Proof of Theorem \ref{maintheorem}}: If $\G^F$ is not one of the groups excluded in Assumption \ref{spaethassumption} then Theorem \ref{main} yields an $\mathcal{H}$-equivariant bijection $f:\Irr_{p'}(B)\to \Irr_{p'}(G)$. Observe that $U$ is a Sylow $p$-subgroup of $G$ and $B= \mathrm{N}_G(U)$. By an inspection of \cite[Table 24.2]{Malle2011} we observe that $\Z(G)=1$ for all groups excluded in Assumption \ref{spaethassumption}. Thus, in this case Corollary \ref{trivialcenter} applies. \qed

%Using Corollary \ref{trivialcenter} we can prove Theorem \ref{main} for some more cases.
%
%\begin{corollary}\label{better}
%	Suppose that $(\G,F)$ is not contained in the table of Theorem \ref{bijection}. If $p \in \{2,3\}$ assume additionally that $q \not \equiv \varepsilon \, \mathrm{mod} \, 3$ if $(\G,F)$ is of type $E_6( \varepsilon q)$ with $\varepsilon \in \{ \pm 1 \}$ and $p \neq 3$ if $\G$ is of type $E_7$. Then there exists an $\mathcal{H}$-equivariant bijection
%	\begin{align*}
%	f:\Irr_{p'}(B) \to \Irr_{p'}(G).
%	\end{align*}
%	%Moreover, for every central character $\lambda\in\Irr(\Z(G))$ the map $f$ restricts to a bijection $\Irr_{p'}(B\mid \lambda)^ \sigma\to \Irr_{p'}(G \mid \lambda)^\sigma$.
%\end{corollary}
%
%\begin{proof}
%	If $p$ is a good prime then this is precisely Theorem \ref{main}. Now suppose that $p$ is a bad prime. By an inspection of \cite[Table 24.2]{Malle2011} we observe that $\Z(G)=1$ unless $q \equiv \varepsilon \, \mathrm{mod} \, 3$ and $(\G,F)$ is of type $E_6( \varepsilon q)$ or $p=3$ and $\G$ is of type $E_7$. If $\Z(G)=1$ we can apply Corollary \ref{trivialcenter}.
%\end{proof}

%NEW STUFF
%

\section{Group automorphisms and Galois automorphisms}\label{section6}

The overall aim of this section is to explicitly construct extensions of suitable $p'$-characters $\chi \in \Irr(G)$ to $GD_{\chi}$. It is known by \cite[Remark 3.4]{Spaeth} that every character $\chi$ extends to $GD_{\chi}$. Unfortunately, the proof given there is only an existence proof and little to no information is given about the extended characters. To prove the inductive Galois-McKay condition for $G$ we need to be able to explicitly compute the action of the stabilizer $( \mathcal{B}  \times \mathcal{H})_{\chi}$ on the extensions of the characters to $GD_\chi$. To do this, we extend the ideas of \cite{Spaeth} and combine them with the method of descending scalars. This yields more natural extensions of $p'$-characters.

%\section{The inductive Galois--McKay condition}

\subsection{Gelfand--Graev characters}
We recall the construction of Gelfand--Graev characters for disconnected reductive groups. Let $\G$ be a connected reductive group and $\tau$ an automorphism of $\G$ commuting with $F$ which stabilizes the $F$-stable pair $(\T,\B)$. In addition, we assume that $\tau$ is a quasi-central automorphism (see \cite[Definition 1.15]{grnoncon}) and has finite order $k$. We consider the reductive group $\G \rtimes \langle \tau \rangle$ and extend $F$ to $\G \rtimes \langle \tau \rangle$ by defining $F(\tau):=\tau$. Recall that $\Gamma_1=\xi^{\G^F}$. We assume that $\tau$ stabilizes the character $\xi$. We let $\hat{\xi} \in \Irr(\U^F \rtimes \langle \tau \rangle)$ be the extension of $\xi$ defined by $\hat{\xi}(\tau):=1$ and denote $\hat{\Gamma}_1={\hat{\xi}}^{\G^F \rtimes \langle \tau \rangle}$. By Mackey's formula it follows that $\hat{\Gamma}_1$ is an extension of $\Gamma_1$. In particular, $\hat{\Gamma}_1$ is multiplicity free.

%We now introduce a certain subset of $\mathrm{Aut}(\G^F \langle \tau \rangle)$.
Let $\phi: \G \langle \tau \rangle \to \G \langle \tau \rangle$ be a bijective morphism which stabilizes $\T \langle \tau \rangle$ and commutes with $F$. Then $\phi$ restricts to an automorphism of $\G^F \langle \tau \rangle$. We denote by $\mathrm{Aut}_0(\G^F \langle \tau \rangle)$ the set of automorphisms of $\G^F \langle \tau \rangle$ obtained in this way.

 The following proposition is a generalization of \cite[Remark 3.4(c)]{Spaeth}.

\begin{proposition}\label{invariantextension}
	Assume the notation as above. Let $\chi \in \Irr(\G^F)$ be a $\tau$-invariant character such that $\D_\G(\chi)$ is a constituent of $\Gamma_1$. Then there exists an extension $\hat{\chi} \in \Irr( \G^F \rtimes \langle \tau \rangle)$ of $\chi$ such that $\hat{\chi}^y= \mu \hat{\chi}$ whenever $y \in \mathrm{Aut}_0(\G^F \langle \tau \rangle) \times \mathcal{G}$ and $\mu \in \Irr( \langle \tau \rangle )$ is a linear character of $p'$-order which satisfy $\hat{\Gamma}_1^y= \mu \hat{\Gamma}_1$.
\end{proposition}

\begin{proof}
	We apply an idea already present in \cite[Remark 3.4(c)]{Spaeth}.
	Denote $\psi:= \D_\G(\chi)$. Since the Gelfand--Graev character $\Gamma_1$ is multiplicity free it follows that there exists a unique extension $\hat{\psi}$ of $\psi$ with $(\hat{\psi}, \hat{\Gamma}_1)\neq 0$. As $\tau$ is a quasi-central automorphism of $\G$ we can apply \cite[Proposition 3.13]{grnoncon} and there exists an extension $\hat{\chi}$ of $\chi$ to $\G^F \rtimes \langle \tau \rangle$ which can be obtained from $\hat{\psi}$ using the isometric involutions $\D_{\G.\tau^j}$, $j=1,\ldots,k$, from \cite[Definition 3.10]{grnoncon}. The map $\D_{\G.\tau^j}$ is defined using Harish--Chandra induction and restriction in the reductive group $\G \rtimes \langle \tau \rangle$, hence it commutes with Galois automorphisms and group automorphisms of $\mathrm{Aut}_0(\G^F \langle \tau \rangle)$. Therefore,
	 $$\D_{\G.\tau^j}(\hat{\Gamma}_1)^y=\D_{\G.\tau^j}(\hat{\Gamma}_1^y)= \D_{\G.\tau^j}(\hat{\Gamma}_1 \mu)=\D_{\G.\tau^j}(\hat{\Gamma}_1 ) \mu,$$
	 where the last equality follows from the first sentence in the proof of \cite[Proposition 3.30]{grnoncon}.
	 Consequently, the character $\hat{\chi}$ satisfies $\hat{\chi}^y= \mu \hat{\chi}$.
\end{proof}

The last proposition can be refined as follows. Let $\chi \in \Irr(\G^F)$ be a $\tau$-invariant character such that $\D_\G(\chi)$ is a constituent of $\Gamma_1$. Since $\chi$ is $\tau$-invariant the character $\theta \in \Irr(Z(G) \mid \chi )$ is $\tau$-invariant as well. We denote by $\xi_\theta \in \Irr(U Z)$ the unique character of $UZ$ extending both $\theta$ and $\xi$. It follows that $\D_\G(\chi)$ is a constituent of $\Gamma_\theta$, where $\Gamma_\theta:=(\xi_\theta)^G$. Note that both $\theta$ and $\xi$ are $\tau$-stable. Therefore, $\xi_\theta$ can be extended to a character $\hat{\xi}_\theta$ of $UZ \langle \tau \rangle$ with $\hat{\xi}_\theta(\tau)=1$. We denote $\hat{\Gamma}_\theta:=(\hat{\xi}_\theta)^{G \langle \tau \rangle}$. The following proposition is then proved in the same way as Proposition \ref{invariantextension}. For completeness we will give a full proof here.

\begin{proposition}\label{invariantextension2}
	Assume the notation as above. Then there exists an extension $\hat{\chi} \in \Irr( \G^F \rtimes \langle \tau \rangle)$ of $\chi$ such that $\hat{\chi}^y= \mu \hat{\chi}$ whenever $y \in \mathrm{Aut}_0(\G^F \langle \tau \rangle) \times \mathcal{G}$ satisfies $\hat{\Gamma}_\theta^y= \mu \hat{\Gamma}_\theta$ for some linear character $\mu \in \Irr( \langle \tau \rangle )$ of $p'$-order. 
\end{proposition}

\begin{proof}
	Denote $\psi:= \D_\G(\chi)$. Since $\Gamma$ is multiplicity free it follows that $\Gamma_\theta$ is multiplicity free as well. It follows that there exists a unique extension $\hat{\psi}$ of $\psi$ with $(\hat{\psi}, \hat{\Gamma}_\theta)\neq 0$. As in the proof of Proposition \ref{invariantextension} we obtain an extension $\hat{\chi}$ of $\chi$ to $\G^F \rtimes \langle \tau \rangle$ which can be obtained from $\hat{\psi}$ using the isometric involutions $\D_{\G.\tau^j}$, $j=1,\ldots,k$. As in the proof of loc. cit. we obtain
	$\D_{\G.\tau^j}(\hat{\Gamma}_{\theta})^y=\D_{\G.\tau^j}(\hat{\Gamma}_{\theta} ) \mu$. Consequently, the character $\hat{\chi}$ satisfies $\hat{\chi}^y= \mu \hat{\chi}$.
\end{proof}

Note that the conclusions of Proposition \ref{invariantextension} and Proposition \ref{invariantextension2} remain true if we replace $\Gamma$ by any multiplicity-free character $\Gamma' \in  \Irr(\G^F)$ with the property that $\Gamma'=(\xi')^{\G^F}$ for a $\tau$-stable linear character $\xi' \in \Irr(\U^F)$.
%\subsection{Extending characters}
% By [Groupes non connexes, Proposition 3.13] there exists a "dual character" $\gamma \in \Irr(\G^F \rtimes \langle \sigma \rangle)$ with $\gamma_{\G^F}=\varepsilon\D_\G(\chi)$.
%We denote by
%$$\D_{\G \rtimes \langle \sigma \rangle}: \Irr(\G^F \rtimes \langle \sigma \rangle) \to \Irr(\G^F \rtimes \langle \sigma \rangle)$$
%given by $\D_{\G \rtimes \langle \sigma \rangle}= \sum_{\B \subseteq \Para} (-1)^{r(\Para^\circ)} R_{\Levi}^{\G} \circ {}^\ast R_{\Levi}^\G$.

\subsection{Action of automorphisms on regular characters}
Let us now assume again that $\G$ is a simple algebraic group of simply connected type. The results of the previous sections suggest that it is important to study the action of Galois automorphisms on Gelfand--Graev characters. We do this by refining the result of Lemma \ref{galoisU}.

\begin{lemma}\label{field of value}
Let $\sigma \in \mathcal{G}$ be a Galois automorphism.
	\begin{enumerate}[label=(\alph*)]
		\item Assume that $(\G,F)$ is untwisted. For every $\sigma \in \mathcal{G}$ there exists some $\tilde{t} \in \tilde{T}^{F_0}$ with $ \phi_S^{\tilde{t}}= \phi_S^\sigma$ for all $S \subseteq \{ 1, \dots, r \}$ and $\gamma(\tilde{t}) \tilde{t}^{-1} \in \mathrm{Z}(\tilde{G})$ for every graph automorphism $\gamma$ of $\G$.
	\item 	Assume that $\G$ is not of type $A_n$, $D_n$ if $n$ is odd and $p \neq 2$. For every $\sigma \in \mathcal{G}$ there exists some $\tilde{t} \in \tilde{T}^{D}$ with $ \phi_S^{\tilde{t}}= \phi_S^\sigma$ for all $S \subseteq \{ 1, \dots, r \}$.
		
%		\item For every $S\subsetneq \{1, \dots, r \}$ there exists $\tilde{t} \in \tilde{T}^D$ such that $\phi_S^\sigma= \phi_S^{\tilde{t}}$.
	\end{enumerate}
\end{lemma}

\begin{proof}
	Let us first assume that $\G^F$ is not isomorphic to $D_4(q)$ or ${}^3 D_4(q)$.
	The automorphism $\sigma \in \mathcal{H}$ sends the element $\phi_0 \in \Irr((\mathbb{F}_{q^w},+))$ again to an element of $\Irr((\mathbb{F}_{q^w},+))$ of order $p$. Hence, there exists some $b \in \mathbb{F}_p^\times$ such that $\phi_0^\sigma(a)=\phi_0(ba)$ for all $a \in \mathbb{F}_{q^w}$. By \cite[Proposition 8.1]{Maslowski2010} the elements $s_i:=\omega_{i}^\vee(b) \in \Ttilde^{F_p}$, $i=1, \dots, n$, satisfy
	%	$\mathrm{Z}( \G)^{F_p}=\langle s_0 \rangle$ and 
	$$x_{\alpha_j}(u)^{s_j}=x_{\alpha_j}(b u) \text{ and } x_{\alpha_i}(u)^{s_j}=x_{\alpha_i}( u)$$
	for all $i,j=1,\dots,n$ and $u\in \mathbf{k}$. We define $\tilde{t}:= \prod_{i=1}^n s_i \in \Ttilde$ and observe that 
	$$ x_{j}(u)^{\tilde{t}}=x_{j}(b u)$$
	for all $j=1,\dots,r$ and $u\in \mathbb{F}_{q^w}$. Recall that for $S \subseteq \{1, \dots, r \}$ we have $\phi_S= \prod_{i \in S} \phi_i$, where $\phi_i(x_i(u))= \phi_0(u)$ for all $u \in \mathbb{F}_{q^{|A_i|}}$. The action of graph automorphisms on the $s_i$ is given by \cite[Corollary 9.2]{Maslowski2010}. From this it follows that $\gamma(\tilde{t}) \tilde{t}^{-1} \in \mathrm{Z}(\Gtilde)$ for every graph automorphism $\gamma$.
	This shows that $F(\tilde{t})\tilde{t}^{-1} \in \mathrm{Z}(\tilde{\G})$ which implies that $\tilde{t}$ normalizes the finite group $\U^F$. By the explicit description of the action of $\tilde{t}$ on the $x_i$'s we easily deduce that $\phi_S^{\tilde{t}}= \phi_S^\sigma$. This proves part (a).
	
	To show part (b) we show using a case-by-case analysis that there exists some $z \in \mathrm{Z}(\tilde{G})$ such that $\tilde{t} z$ is $D$-stable. If $\G$ is of type $D_n$ with even $n$ then $\tilde{t}$ is already $D$-stable.
	Now suppose that $\G$ is of type $A_n$ with $n$ even or of type $E_6$. Then we have $\gamma(\tilde{t}) \tilde{t}^{-1}=s_0^{-n}$, where $s_0$ is the generator of $\mathrm{Z}(\Gtilde)^{F_p}$ as in \cite[Proposition 8.1]{Maslowski2010}. It follows that the element $z\tilde{t}$ with $z:=s_0^{-\frac{n}{2}}$ is $D$-stable.

	Finally assume that $\G^F$ is isomorphic to $D_4(q)$ or ${}^3 D_4(q)$. By Lemma \ref{galoisU} there exists $\tilde{t} \in \tilde{T}$ such that $\xi^\sigma=\xi^{\tilde{t}}$. Let $X$ denote the group of graph automorphism of $\G$. Since $\xi$ is $X$-stable we have $\gamma(\tilde{t})\tilde{t}^{-1} \in \tilde{T}_\xi=T \mathrm{Z}(\tilde{G})$ for all $\gamma \in X$. Therefore, the image of $\tilde{t}$ in $ H^1(F,\mathrm{Z}(\G))$ is $X$-stable. Since $\mathrm{Z}(\G)^X=1$ we observe that $H^1(F,\mathrm{Z}(\G))^X=1$ and so $\tilde{t} \in T\mathrm{Z}(\tilde{G})$. Consequently, $\xi^\sigma=\xi$ and thus $\phi_S^\sigma=\phi_S$ for all $\sigma \in \mathcal{G}$.
%	Now assume that $S$ is a proper subset of $\{ 1, \dots, r\}$. We assume that $j \notin S$ and let $i \in A_j$. Then the description in \cite[Corollary 9.2]{Maslowski2010} shows that there exists some integer $x$ such that $\tilde{t} s_i^x$ is $D$-stable. Observe that $\tilde{t} s_i^x$ acts as $\tilde{t}$ on $\U_{A_j}^F$, $j \neq i$. Since $\phi_S$ is trivial on the subgroup $\U_{A_i}^F$ it follows that $\phi_S^{\tilde{t} s_i^x}= \phi_S^{\tilde{t}}= \phi_S^\sigma$. This proves our claim.
	% WHAT ABOUT E_6????? ----> HAVE TO LOOK AT TABLE 6.9
	%	Now assume that $F$ is a twisted Frobenius. In this case the group $D$ is generated by the field automorphism $F_p$.
	%	By \cite[Proposition 10.2]{Maslowski2010} the elements $s_i=N_{F^w/F}(\omega^\vee_{a_i})(b) \in \Ttilde^{}$, $i=1, \dots, r$ satisfy
	%%	$\mathrm{Z}( \G)^{F_p}=\langle s_0 \rangle$ and 
	%	$$x_{j}(u)^{s_i}=x_{j}(b^{\frac{p^w-1}{p^{|A_i|}-1}} u)$$
	%	for all $i,j=1,\dots,r$ and $u\in \mathbb{F}_q$. We define $\tilde{t}:= \prod_{i=1}^r s_i$ and observe that 
	%	$$ x_{j}(u)^{\tilde{t}_0}=x_{j}(b u)$$
	%	for all $j$ and $u\in \mathbb{F}_q$. From this we conclude that $ \phi_S^{\tilde{t}}= \phi_S^\sigma$ for all $S \subseteq \{ 1, \dots, r \}$.
	%	
\end{proof}

The restriction on the type of the group of Lie type in Lemma \ref{field of value}(b) seems to be necessary in general. However, the following result is true in general.
%
%A consequence of the previous lemma is that $F_{p^e}\tilde{t}$ normalizes the group $G D_\psi$.
\begin{lemma}\label{field of value: better}
If $\sigma \in \mathcal{G}$ is a Galois automorphism then there exists some $\tilde{t} \in \tilde{T}$ such that $\phi_S^{\tilde{t}}=\phi_S^{\sigma}$ for all $S \subseteq \{1,\dots, r\}$ and $d(\tilde{t}) \tilde{t}^{-1} \in \mathrm{Z}(\tilde{G})$ for all $d\in D$.
\end{lemma}

\begin{proof}
	In the untwisted case this is a consequence of Lemma \ref{field of value}(a). By Lemma \ref{field of value}(b) we can assume that $\G$ is not of exceptional type. Suppose now that $F$ is a twisted Frobenius endomorphism. If $q$ is a square then \cite[Theorem 1.8(i)]{TiepZalesski} shows that $\phi_S^\sigma=\phi_S$ for all $S$. Hence, the statement follows in this case. 
	
	If $q=p^f$ is not a square, i.e. $f$ is odd, and $F=F_q \gamma$, then $(F_p \gamma)^f=F$. Denote $F':=F_p \gamma$ and let $\tilde{t} \in \tilde{\T}$ the element constructed in Lemma \ref{field of value}. Then we have $\phi_S^\sigma=\phi_S^{\tilde{t}}$ and $F'(\tilde{t}) \tilde{t}^{-1} \in \mathrm{Z}(\Gtilde)$. Applying Lang's theorem to the Frobenius endomorphism $F': \mathrm{Z}(\Gtilde) \to \mathrm{Z}(\Gtilde)$ we observe that there exists some $z \in \mathrm{Z}(\Gtilde)$ such that $\tilde{t} z$ is $F'$-stable. We can replace $\tilde{t}$ by $\tilde{t} z$ and therefore assume that $\tilde{t} \in \Ttilde^{F'}=\tilde{T}^{F'}$. Furthermore, we still have $\phi_S^{\tilde{t}}= \phi_S^\sigma$ since we only changed $\tilde{t}$ by a central element. The same reasoning implies that we still have $\gamma(\tilde{t}) \tilde{t}^{-1} \in \mathrm{Z}(\tilde{G})$. This yields the statement of the lemma.
\end{proof}

\subsection{Descent of scalars}

In this section, we suppose that $F_0$ is a Frobenius endomorphism with $F_0^k=F$ for some integer $k$.
For any $F$-stable closed subgroup $\mathbf{H}$ of $\G$ we denote $$\underline{\mathbf{H}}:= \mathbf{H} \times F^{k-1}_0(\mathbf{H}) \times \ldots \times F_0(\mathbf{H}).$$
We consider the automorphism
$$\tau: \underline{\G} \to \underline{\G}$$
with $\tau(g_1,\dots,g_k)=(g_2,\dots,g_k,g_1)$.
The mapping $\mathbf{H} \mapsto \underline{\mathbf{H}}$ yields a bijection between closed $F$-stable subgroups of $\G$ and $\tau F_0$-stable closed subgroups of $\underline{\G}$. For any such $\mathbf{H}$, the projection map $\mathrm{pr}:\underline{\mathbf{H}} \to \mathbf{H}$ onto the first coordinate yields an isomorphism $\underline{\mathbf{H}}^{\tau F_0} \cong \mathbf{H}^F$. 

If $\phi: \G \to \G$ is a bijective morphism we denote by $\underline{\phi}: \underline{\G} \to \underline{\G}$ the bijective morphism given by $\underline{\phi}(g_1,\dots,g_k):=(\phi(g_1),\dots, \phi(g_k))$. For convenience we also write $F_0$ for the Frobenius endomorphism $\underline{F_0}: \underline{\G} \to \underline{\G}$. One easily verifies that the automorphism $\tau$ is a quasi-central automorphism of $\underline{\G}$ in the sense of \cite[Definition 1.15]{grnoncon}.

 We consider the non-connected reductive group $\underline{\G} \rtimes \langle \tau \rangle$ with Frobenius endomorphism $\tau F_0: \underline{\G} \to \underline{\G}$. One easily checks that $\mathrm{pr}$ induces an isomorphism 
 $$\underline{\G}^{\tau F_0} \rtimes \langle  \tau \rangle \cong \G^F \rtimes \langle  F_0 \rangle.$$

\subsection{Construction for twisted groups}

We generalize the construction of the previous section. This is essentially necessary for working with automorphisms of twisted groups.
We suppose now that $F_0$ is a Frobenius endomorphism with $F_0^k \rho=F$ for some integer $k$ and $\rho: \G \to \G$ a graph automorphism of order $l$ which commutes with $F_0$.
For any $F^{kl}$-stable closed subgroup $\mathbf{H}$ of $\G$ we denote $$\underline{\mathbf{H}}:= \mathbf{H} \times F^{kl-1}_0(\mathbf{H}) \times \ldots \times F_0(\mathbf{H}).$$
As before, we consider the automorphism
$$\tau: \underline{\G} \to \underline{\G}, \,  (g_1,\dots,g_{kl}) \mapsto (g_2,\dots,g_{kl},g_1)$$
%The mapping $\mathbf{H} \mapsto \underline{\mathbf{H}}$ yields a bijection between closed $F^l$-stable subgroups of $\G$ and $\tau F_0$-stable closed subgroups of $\underline{\G}$.
Recall that for any such $\mathbf{H}$, the projection map $\mathrm{pr}:\underline{\mathbf{H}} \to \mathbf{H}$ onto the first coordinate yields an isomorphism $\underline{\mathbf{H}}^{\tau F_0} \cong \mathbf{H}^{F^{kl}}$. 
%For any integer $j$ we define $g_j:=g_{j'}$, where $1 \leq j' \leq lk$ and $j \equiv j' \, \mathrm{mod} lk$.
Let us consider the closed subgroup $$\underline{\G}_{\rho}:=\{ \underline{g} \in \underline{\G} \mid g_i=\rho(g_{k+i}) \text{ for } i=1,\dots k(l-1) \}$$
of $\underline{\G}$. In fact, one easily observes that $\underline{\G}_{\rho}$ is isomorphic to the $k$-fold product of $\G$. In particular, $\underline{\G}_{\rho}$ is a connected reductive group. For any subgroup $\mathbf{H}$ of $\underline{\G}$ we define $\mathbf{H}_\rho:=\mathbf{H} \cap \underline{\G}_\rho$.
The fundamental observation is the following:

\begin{lemma}
The projection map $\mathrm{pr}: \underline{\mathbf{G}}^{\tau F_0} \to \mathbf{G}^{F_0^{kl}}$ maps $\underline{\G}^{\tau F_0}_{\rho}$ to $\G^F$. Moreover, the projection map induces an isomorphism
	$$\underline{\G}_{\rho}^{\tau F_0} \rtimes \langle  \tau \rangle \cong \G^{F} \rtimes \langle F_0 \rangle.$$
\end{lemma}

\begin{proof}
	Assume first that $\underline{g}=(g_1,\dots,g_{kl}) \in \underline{\G}_\rho^{\tau F_0}$. Then we have $\underline{g}=(g_1,F_0^{kl-1}(g_1), \dots, F_0(g_1))$. Since $\underline{g} \in \underline{\G}_\rho^{\tau F_0}$ we have $g_1=\rho(g_{k+1}) =\rho(F_0^k(g_1))=F(g_1)$. In other words, $\mathrm{pr}(\underline{g})=g_1$ is $F$-stable. On the other hand, if $g \in \G^F$ then $\rho(F_0^k(g))=F(g)=g$ which shows that $\mathrm{pr}^{-1}(g)=(g,F_0^{kl-1}(g), \dots, F_0(g)) \in \underline{\G}^{\tau F_0}_{\rho}$. Therefore, $\underline{\G}^{\tau F_0}_{\rho} \cong \G^F$ as required. Furthermore, we know that $\mathrm{pr}$ induces an isomorphism 
	$$\underline{\G}^{\tau F_0} \rtimes \langle \underline{\rho} \tau \rangle \cong \G^{F^l} \rtimes \langle F_0 \rangle.$$
	This isomorphism then restricts to an isomorphism $\underline{\G}_{\rho}^{\tau F_0} \rtimes \langle  \tau \rangle \cong \G^{F} \rtimes \langle F_0 \rangle.$
\end{proof}
%If $\phi: \G \to \G$ is a bijective morphism we denote by $\underline{\phi}: \underline{\G} \to \underline{\G}$ the bijective morphism given by $\underline{\phi}(g_1,\dots,g_k):=(\phi(g_1),\dots, \phi(g_k))$. For convenience we also write $F_0$ for the Frobenius endomorphism $\underline{F_0}: \underline{\G} \to \underline{\G}$. One easily verifies that the automorphism $\underline{\rho} \tau$ is a quasi-central automorphism of $\underline{\G}$ in the sense of \cite[Definition 1.15]{grnoncon}.

\subsection{Character-theoretic consequences}

We fix a character $\psi_0 \in \Irr_{p'}(B \mid \phi_S) \cap \mathcal{T}$ and denote $\chi_0:=f( \psi_0)$, where $f: \Irr_{p'}(B) \to \Irr_{p'}(G)$ is as in Theorem \ref{main}. For the remainder of this section we assume that $\chi_0$ is $F_0$-stable, where $F_0: \G \to \G$ is (as in the previous section) a Frobenius endomorphism which satisfies $F_0^k= \rho F$ for some graph automorphism $\rho$ of order $l$ which commutes with $F_0$. In particular, this implies that $\rho$ stabilizes the character $\chi_0$. Let $\theta \in \Irr(\mathrm{Z}(G))$ be the common central character of $\psi_0$ and $\chi_0$.

\begin{assumption}\label{tilde}
We suppose that for every $\sigma \in \mathcal{H}$ there exists an element $\tilde{t} \in \tilde{T}$ such that $d(\tilde{t}) \tilde{t}^{-1} \in \mathrm{Z}(G)$ for all $d \in D_{\chi_0}$ and $\xi^{\tilde{t}}= \xi^\sigma$. If $\G^F\cong D_4(q)$ we let $\tilde{t}=1$, which is possible by the proof of Lemma \ref{field of value}.
\end{assumption}

We remark that the existence of such an element has not been shown in general, see Lemma \ref{field of value}. Until the end of this section, we will assume that Assumption \ref{tilde} holds. We then define $x_{\sigma}:=F_{p^e} \sigma^{-1} \tilde{t}$.

Recall that $\xi_\theta \in \Irr(U \mathrm{Z}(G))$ denotes the common extension of $\xi \in \Irr(U)$ and the central character $\theta \in \Irr(\mathrm{Z}(G))$. The character $\xi_\theta$ has an extension $\hat{\xi}_\theta \in \Irr(U \mathrm{Z}(G) D_{\chi_0} )$ with $D_{\chi_0}$ in its kernel. Therefore, there exists a linear character $\mu \in \Irr(D_{\chi_0})$ (which possibly depends on the Galois automorphism $\sigma$ and $\theta$) such that $\hat{\xi}_\theta^{x_\sigma}=\hat{\xi}_\theta \mu$. Evaluating this equality at $d \in D_{\chi_0}$ yields 
$$\mu(d)= \theta^{F_{p^e} \sigma^{-1}}(d(\tilde{t}) \tilde{t}^{-1}).$$
Since $\theta \in \Irr(\mathrm{Z}(G))$ is a character of a group of $p'$-order we conclude that the linear character $\mu$ is always of $p'$-order.

\begin{proposition}\label{invariant extension}
Suppose Assumption \ref{tilde} and let $\mu \in \Irr(D_{\chi_0})$ be the character defined above. Then there exists an extension $\hat{\chi_0} \in \Irr(\G^F \langle F_0 \rangle)$ of $\chi_0$ which satisfies $\hat{\chi}_0^{x_\sigma} = \hat{\chi}_0 \mu$ for every $\sigma \in \mathcal{H}$.
\end{proposition}

\begin{proof}
	Let $\mathrm{pr}^\vee: \mathbb{Z} \Irr(\G^F) \to \mathbb{Z} \Irr(\underline{\G}_\rho^{\tau F_0})$ be the map induced by $\mathrm{pr}$. We let $\underline{\chi}_0:=\mathrm{pr}^\vee( \chi_0) \in \Irr_{p'}( \underline{\G}_\rho^{\tau F_0})$.
%	 be the character of $\underline{\G}^{\tau F_0}$ which corresponds to $\chi_0$ under the isomorphism $\underline{\G}^{\tau F_0} \cong \G^F$ given by $\mathrm{pr}$.
%	 Furthermore, let $\xi \in \Irr( \U^F)$ be the character such that $\Gamma_1= \xi^{\G^F}$.
	 Observe that $(\underline{\U}\Z(\underline{\G}))_{\rho}^{\tau F_0}=\underline{\U}_\rho^{\tau F_0} \Z(\underline{\G})_{\rho}^{\tau F_0}$.
	 We let $\underline{\xi}_{\underline{\theta}} \in \Irr( (\underline{\U}\Z(\underline{\G}))_{\rho}^{\tau F_0} )$ be the character of $(\underline{\U}\Z(\underline{\G}))_{\rho}^{\tau F_0} $ which corresponds to $\xi_\theta$ under the isomorphism $(\underline{\U}\Z(\underline{\G}))_{\rho}^{\tau F_0}\cong \U^F \mathrm{Z}(\G)^F$ given by $\mathrm{pr}$. Define $\underline{\Gamma}_\theta:= \underline{\xi}_\theta^{\underline{\G}_\rho^{\tau F_0}}$, which is the character corresponding to $\Gamma_\theta$ under the isomorphism $\underline{\G}_\rho^{\tau F_0} \cong \G^F$ given by $\mathrm{pr}$. The projection map $\mathrm{pr}$ induces an isomorphism $\underline{\G}_\rho^{\tau F_0} \cong \G^F$ and maps the BN-pair $(\underline{\B}_\rho^{\tau F_0}, \mathrm{N}_{\underline{\G}_\rho^{\tau F_0}}(\underline{\T}_\rho) )$ of $\underline{\G}_\rho^{\tau F_0}$ to the BN-pair $(\B^F, \mathrm{N}_{\G^F}(\T))$ of $\G^F$. Since Alvis--Curtis duality depends only on the BN-pair structure of the group (see \cite[Section 8.2]{Carter1985}) we deduce that $\mathrm{D}_{\underline{\G}_\rho} \circ \mathrm{pr}^\vee = \mathrm{pr}^\vee \circ \mathrm{D}_{\G}$. 
		
		From this we deduce that $\mathrm{D}_{\underline{\G}_\rho}(\underline{\chi}_0)$ is a constituent of $\underline{\Gamma}_\theta$. Let $\phi: \G \to \G$ be the bijective morphism given by the action of $F_{p^e} \tilde{t}$. By Assumption \ref{tilde}, the morphism $\phi$ commutes with $\rho$. Therefore, we can consider the restriction of $\underline{\phi}$ to $\underline{\G}_\rho$, which we will denote by the same letter. Observe that $\underline{\Gamma}_\theta^{\underline{\phi} \sigma^{-1}}=\underline{\Gamma}_\theta \mathrm{pr}^\vee(\mu)$.
		The remarks after Proposition \ref{invariantextension2} show that there
		exists an extension $\hat{\underline{\chi}}_0 \in \Irr( \underline{\G}_\rho^{\tau F_0} \langle \tau \rangle )$ of $\underline{\chi}_0$ which satisfies $\hat{\underline{\chi}}_0^{\underline{\phi} \sigma^{-1}}=\hat{\underline{\chi}}_0 \mathrm{pr}^\vee(\mu)$. We conclude that the character $\hat{\chi_0}:=(\mathrm{pr}^\vee)^{-1}(\hat{\underline{\chi}}_0) \in \Irr_{p'}(\G^F \langle F_0 \rangle)$ is an extension of $\chi_0$ which satisfies $\hat{\chi}_0^{x_\sigma} = \hat{\chi}_0 \mu$.
\end{proof}

\subsection{Gluing extensions of characters}

We need the following elementary lemma about extending characters to semidirect products.

\begin{lemma}\label{direct product}
	Let $Y$ be a finite group and suppose that $X:=X_1 \times X_2$ acts on $Y$, where $X_1$, $X_2$ are abelian groups. Let $\chi \in \Irr(Y)$ be an irreducible character which extends to $Y \rtimes X$ and let $\chi_i$ be extensions of $\chi$ to $Y \rtimes X_i$ for $i=1,2$. Then the character $\chi$ has an extension $\tilde{\chi}$ to $Y \rtimes X$ such that whenever $z \in \mathcal{G} \times  \mathrm{N}_{\mathrm{Aut}(Y \rtimes X)}(YX_1,Y X_2)$ and $\mu \in \Irr(X)$ satisfy $ \chi_i^{z} =\mu_{X_i} \chi_i$ for $i=1,2$ then we have $ \tilde{\chi}^{z}= \mu \tilde{\chi}$.
\end{lemma}

\begin{proof}
	%	It is clear that if $\chi$ has an $a \sigma$-invariant extension to $Y \rtimes A$ then its restriction to $Y \rtimes A_i$ is an $a \sigma$-invariant extension of $\chi$ for $i=1,2$.
	
	By assumption there exists an extension $\tilde{\chi}$ of $\chi$ to $Y \rtimes X$. By \cite[Lemma 13.21]{Digne1991} there exist linear characters $\lambda_i \in \Irr(Y \rtimes X_i / Y)$ such that $\chi_i= \lambda_i \tilde{\chi}_{Y \rtimes X_i}$. Let $\lambda \in \Irr(Y \rtimes X / Y)$ be the linear character defined by $\lambda((x_1,x_2)):= \lambda_1(x_1) \lambda_2(x_2)$ for $(x_1, x_2) \in X_1 \times X_2$. Define $\tilde{\chi}':= \lambda \tilde{\chi}$.
	
	The characters $\chi_i \in \Irr(Y \rtimes X_i \mid \chi )$ satisfy $\chi_i^{z } =\mu_{X_i} \chi_i$ by assumption. We claim that $(\tilde{\chi}')^{z}= \mu \tilde{\chi}'$. Since $\chi$ is $z$-invariant it follows that there exists some linear character $\mu' \in \Irr(X)$ such that $ (\tilde{\chi}')^{z }=\mu' \tilde{\chi'} $. To show that $\mu= \mu'$ it suffices to show that $\mu'_{X_i}= \mu_{X_i}$ for $i=1,2$. However $(\tilde{\chi}')_{Y \rtimes X_i} = \chi_i$ and therefore $ \chi_i \mu_{X_i}= \chi_i^{z}= \chi_i \mu_{X_i}$. Hence, by \cite[Lemma 13.21]{Digne1991} we have $\mu_{X_i}'=\mu_{ X_i}$. Consequently, we have $ (\tilde{\chi}')^{z }= \tilde{\chi'}$.
\end{proof}

%\begin{assumption}\label{AD}
%	We assume from now on that $(\G,F)$ is not of type $D_4$ or of twisted type.
%\end{assumption}

%The proof of the following corollary relies heavily on our assumption that $(\G,F)$ is untwisted.

\begin{corollary}\label{invariant global}
Suppose that Assumption \ref{tilde} holds. Then there exists an extension $\hat{\chi}_0 \in \Irr(G D_{\chi_0})$ of $\chi_0$ which satisfies $\hat{\chi}_0^{x_\sigma} = \hat{\chi}_0 \mu$ for every $\sigma \in \mathcal{H}$.
\end{corollary}

\begin{proof}
	Assume first that $D_{\chi_0}$ is cyclic. Then there exists a Frobenius endomorphism $F_0$ of $\G$ with $F_0^k=\gamma F$ for some integer $k$ and a (possibly trivial) graph automorphism $\gamma$ such that $F_0$ generates $D_{\chi_0}$. Then the claim follows immediately from Proposition \ref{invariant extension}. Let us now assume that $D_{\chi_0}$ is a non-cyclic abelian group. Consequently, there exists a field automorphism $F_0$ with $F_0^k=F$ for some integer $k$ and a graph automorphism $\gamma$ such that $D_{\chi_0}=\langle \gamma, F_0 \rangle$. By \cite[Remark 3.6]{Spaeth} the character $\chi_0$ extends to $G D_{\chi_0}$. Applying Proposition \ref{invariant extension} yields an extension of $\chi_0$ to $G \langle F_0 \rangle$. Furthermore, Proposition \ref{invariantextension} gives an extension of $\chi_0$ to $G \langle \gamma \rangle$. Now applying Lemma \ref{direct product} shows that the character $\chi_0$ has an extension to $GD_{\chi_0}$ which satisfies the required property.
	
	Finally assume that $D_{\chi_0}$ is non-abelian. By an inspection of the automorphism groups of groups of Lie type we easily see that $G\cong D_4(q)$. In this case, $\tilde{t}=1$ and consequently $\mu=1_{D_{\chi_0}}$. Let $D_1$ be the normal Sylow $3$-subgroup of $D_{\chi_0}$. An analysis of the subgroup structure of $D$ shows that there exists some abelian subgroup $D_2$ of $D_{\chi_0}$ such that $D_{\chi_0}=D_1 D_2$ and $D_1 \cap D_2= 1$. Since $F$ is untwisted we can write $D_i=\langle \gamma_i,F_{0,i} \rangle$ where the $\gamma_i$'s are (possibly trivial) graph automorphisms and $F_{0,i}^{k_i}=F$ for some $k_i$'s. We can therefore apply the arguments from before to the groups $D_i$ and conclude that there exists an $x_\sigma$-equivariant extension $\chi_i \in \Irr(G D_i)$ of $\chi_0$. Furthermore, using the arguments used in Proposition \ref{invariantextension} and Proposition \ref{invariant extension} one can show that the characters $(\chi_1)_{G \langle \gamma_1 \rangle}$ and $(\chi_1)_{{G \langle F_{0,1} \rangle}}$ are $D_2$-stable as well. By Lemma \ref{direct product}, we conclude that $\chi_1$ is $D_2$-stable. According to \cite[Lemma 2.11]{Spaeth} we obtain a character $\hat{\chi}_0 \in \Irr(G {D_{\chi_0}})$ such that the restriction of $\hat{\chi}_0$ to $GD_i$ is $\chi_i$. By \cite[Lemma 13.21]{Digne1991} there exists some $\lambda \in \Irr({D_{\chi_0}})$ such that $\hat{\chi}_0^{x_\sigma}=\hat{\chi}_0 \lambda$. Since $(\hat{\chi}_0)_{GD_i}=\chi_i$ is $x_\sigma$-invariant, we deduce that $\lambda_{D_i}=1_{D_i}$ whence $\lambda=1_{D_{\chi_0}}$. This proves that $\hat{\chi}_0$ is $x_\sigma$-stable.
\end{proof}

%\begin{remark}\label{twisted}
%	Assume that $F=F_q \gamma$ for some graph automorphism $\gamma$ of order $2$. Then $D$ is generated by $\langle \gamma, F_p \gamma \rangle$. If $q=p^l$ for an odd integer $l$ then $(F_p \gamma)^l=F$. In this case we find that $D_{\chi_0}= \langle F_0, \rho \rangle$ for some Frobenius endomorphism $F_0$ with $F_0^k=F$ for some $k$ with $k \mid l$ and some (possibly trivial) graph automorphism $\rho$. The conclusion of Corollary \ref{invariant global} therefore remains valid in this case. If $k$ is even this is however not possible. On the other hand, if $k$ is even then every unipotent element of $\G^F$ is rational according to \cite[Theorem 1.8(i)]{TiepZalesski}.
%\end{remark}

The local analogue of Corollary \ref{invariant global} is easier to prove.

\begin{proposition}\label{invariant extension local}
Suppose that Assumpton \ref{tilde} holds. Let $\psi \in \Irr_{p'}(\tilde{B}^F)$ and assume that $\tilde{t} \in \tilde{T}$ satisfies $d(\tilde{t})\tilde{t}^{-1} \in \mathrm{Z}(G)$ for all $d \in D_{\chi_0}$. Then the character $\psi_0 \in \Irr(B \mid \psi)$ constructed in the proof of Lemma \ref{local} has an extension $\hat{\psi}_0 \in \Irr(B D_\psi )$ which satisfies $\hat{\psi}_0^{x_\sigma} = \hat{\psi}_0 \mu$ for every $\sigma \in \mathcal{H}$.
\end{proposition}

\begin{proof}
 Let us briefly recall the construction of $\psi_0$ from Lemma \ref{local}. By Clifford correspondence there exists a unique character $\lambda \in \Irr( \tilde{B}_{\phi_S} \mid \phi_S)$ such that $\lambda^{\tilde{B}}= \psi$. Then $\psi_0 \in \Irr(B \mid \phi_S)$ is defined as $\psi_0= (\lambda_{B_{\phi_S}})^B$.
 
 Denote $E:= D_{\psi_0}$ and $I:=B_{\phi_S}$. By \cite[Proposition 8.4]{Maslowski2010}, the character $\phi_S$ is $E$-stable. Hence by Clifford correspondence, the character $\lambda_I$ is $E$-stable as well.
 
  The group $\tilde{B}_{\phi_S}/ \mathrm{Ker}(\phi_S)$ is abelian by the proof of \cite[Lemma 8.5]{Maslowski2010}. We deduce that $I/\mathrm{Ker}(\lambda_I)$ is abelian as well. Thus, we can consider $\lambda_I$ as a character of the abelian group $I/\mathrm{Ker}(\lambda_I)$. Since $I E / \mathrm{Ker}(\lambda_I)\cong I/\mathrm{Kern}(\lambda_I) \rtimes E$ is the semidirect product with an abelian normal factor we know that the $E$-stable character $\lambda_I$ extends to $I E$. In particular, there exists an extension $\hat{\lambda} \in \Irr( I E)$ of $\lambda_I$ with $E$ in its kernel. By Mackey's formula the character $\hat{\psi}_0:=\hat{\lambda}^{ B E}$ is an extension of the character $\psi_0$.
 
It remains to show that $\hat{\psi}_0$ satisfies the desired property. We have $\phi_S^\sigma=\phi_S^{\tilde{t}}$. In particular, the character $\phi_S$ is $x_{\sigma}$-stable. Since $\psi_0$ is $x_\sigma$-stable it follows that its Clifford correspondent $\lambda  \in \Irr( \tilde{B}_{\phi_S} \mid \phi_S)$ is $x_\sigma$-stable as well. Consequently, $\hat{\lambda}$ and $\hat{\lambda}^{x_\sigma}$ are both extensions of $\lambda_I$. Therefore, there exists a linear character $\mu' \in \Irr(E)$ such that $\hat{\lambda}^{x_\sigma}=\mu' \hat{\lambda}$. Since $\hat{\psi}=\hat{\lambda}^{ B E}$ is an extension of $\psi$, it follows that $\theta \in \Irr(\mathrm{Z}(G) \mid \hat{\lambda})$. Thus, evalutation of the equation $\hat{\lambda}^{x_\sigma}=\mu' \hat{\lambda}$ at $d \in E$ yields $\mu'(d)= \theta^{F_{p^e} \sigma^{-1}}(d(\tilde{t}) \tilde{t}^{-1})=\mu(d)$. It follows that $\mu=\mu'$ and therefore $\hat{\psi}_0^{x_\sigma} = \hat{\psi}_0 \mu$.
\end{proof}

\section{The inductive Galois--McKay condition}\label{section7}

\subsection{Projective representations}\label{proj rep}
In this subsection, we prove the inductive Galois--McKay condition for groups of Lie type in defining characteristic. For this we need to recall the statement of \cite[Lemma 1.4]{NSV}:

\begin{lemma}\label{function}
Let $X$ be a finite group and $Y \lhd X$. Let $\theta \in \Irr (Y)$ and assume that
$\theta^{g\sigma }=\theta$ for some $g\in X$ and
$\sigma \in \mathcal{G}$.
Let $\mathcal{P}$ be a projective representation of $X_\theta$
associated with $\theta$ with values in $\mathbb{Q}^{\mathrm{ab}}$ and factor set $\alpha$.
Then $\mathcal{P}^{g\sigma}$ is a projective representation
%=\mathcal{P}(gxg^{-1})^\sigma,$$ where we apply $\sigma$ defines
%a projective representation of $X_\theta$ 
associated with $\theta$, with factor set $\alpha^{g\sigma}(x, y)=\alpha^g(x,y)^\sigma$ for $x,y \in X_\theta$.
In particular, there exists a unique function $$\mu_{g\sigma} : X_\theta \to K^\times$$
with $\mu_{g\sigma}(1)=1$, constant on cosets of $Y$
such that the projective representation $\mathcal{P}^{g\sigma}$ is similar to $\mu_{g\sigma} \mathcal{P}$.
\end{lemma}

Let $\tilde{X} \in \{ \tilde{G}, \tilde{B} \}$ and $X:= G \cap \tilde{X}$. Fix a character $\psi \in \Irr( \tilde{X})$ and let $\psi_0 \in \Irr(X \mid \psi )$ be the characters of $X$ considered in Lemma \ref{global} respectively Lemma \ref{local}. Denote by $\psi_1 \in \Irr(\tilde{X}_\psi \mid \psi_0)$ the Clifford correspondent of $\psi$. We assume without loss of generality that all linear representations are realized over the field $\mathbb{Q}^{\mathrm{ab}}$.

Let $\mathcal{D}$ be a representation affording $\psi_0$. Let $\mathcal{D}_1$ be the representation of $\tilde{X}_{\psi_0}$ affording $\psi$ and extending $\mathcal{D}$. Furthermore, by \cite[Remark 3.4]{Spaeth} and \cite[Remark 3.6]{Spaeth} there exists a representation $\mathcal{D}_2$ of $X D_{\psi_0}$ extending $\mathcal{D}$. For $i=1,2$ we let $\psi_i$ be the character of the representation $\mathcal{D}_i$. As in \cite[Lemma 2.11]{Spaeth} we consider the projective representation $\mathcal{P}$ of $(\tilde{X}D)_{\psi_0}$ defined by 
$$\mathcal{P}(\tilde{g} d):= \mathcal{D}_1(\tilde{g}) \mathcal{D}_2(d)$$
for $\tilde{g} \in \tilde{X}_{\psi_0}$ and $d \in X D_{\psi_0}$. We can now state the following lemma:

\begin{lemma}\label{compute mu}
	Let $y\in \mathcal{H} \times \mathrm{N}_{\mathcal{B}}(X D_{\psi_0} )$ with $ \psi_0^y= \psi_0$. Suppose that $\mu_1 \in \Irr( \tilde{X}_{\psi_0} /X)$ and $\mu_2 \in \Irr( X D_{\psi_0} / X)$ are such that $ \psi_i^y = \mu_i \psi_i$ for $i=1,2$. Then there exists an invertible matrix $M$ such that
	$$ \mathcal{P}^y( \tilde{g} d)= \mu_1(\tilde{g}) \mu_2(d) M \mathcal{P}(\tilde{g} d) M^{-1}$$ for all $\tilde{g} \in \tilde{X}_{\psi_0}$ and $d \in X D_{\psi_0}$.
\end{lemma}

\begin{proof}
	The character $\psi_0$ is $y$-stable, hence there exists an invertible matrix $M$ such that $ \mathcal{D}^y=M \mathcal{D} M^{-1}$. Note that by Schur's lemma the matrix $M$ is determined up to a scalar. Since $ \mathcal{D}_1^y$ is a second extension of $\mathcal{D}$ to $\tilde{X}_{\psi_0}$ there exists by \cite[Lemma 13.21]{Digne1991} an invertible matrix $S$ such that $\mathcal{D}^y_1= \mu_1 S \mathcal{D}_1 S^{-1}$. Hence, we obtain $ \mathcal{D}^y = S \mathcal{D} S^{-1}$. Consequently, $S= \lambda M$ for some scalar $\lambda \in K^\times$, so that we may assume that $S=M$. In particular, we have $\mathcal{D}^y_1= \mu_1 M \mathcal{D}_1 M^{-1}$.
	
	Moreover, note that $ \mathcal{D}^y_2$ is a second extension of $\mathcal{D}$ to $X D_{\psi_0}$. The same reasoning as above shows that $ \mathcal{D}^y_2= \mu_2 M \mathcal{D}_2 M^{-1}$. From this we deduce that $$ \mathcal{P}^y( \tilde{g} d)= \mu_1(\tilde{g}) \mu_2(d) M \mathcal{P}(\tilde{g} d) M^{-1}$$ for all $\tilde{g} \in \tilde{X}_{\psi_0}$ and $d \in X D_{\psi_0}$.
\end{proof}

\subsection{Verifying the inductive condition} We are now able to verify the inductive Galois--McKay condition for most groups of Lie type in defining characteristic.

\begin{theorem}\label{main2}
Suppose that $(\G,F)$ satisfies Assumption \ref{spaethassumption} and assume additionally that $\G$ is not of type $A_n$, $D_n$ for odd $n$. Assume that $F: \G \to \G$ is an untwisted Frobenius endomorphism such that $S:=\G^F/ \mathrm{Z}(\G^F)$ is a simple non-abelian group and $\G^F$ is its universal covering group. Then the inductive Galois--McKay condition from \cite[Definition 3.1]{NSV} holds for the group $S$ and $p$.
\end{theorem}

\begin{proof}
By Theorem \ref{main} there exists an $\mathcal{H} \times \mathcal{B}$-equivariant bijection $f:\Irr_{p'}(B)\to \Irr_{p'}(G)$. Let $\chi_0 \in \mathcal{T}_0$ and $\psi_0:=f(\chi_0)$. Fix a character $\chi \in \Irr(\tilde{G})$ lying above $\chi_0$ and set $\psi:=\tilde{f}( \chi)$. The definition of $\mathcal{H}$-character triples is given in \cite[Definition 1.5]{NSV}. By \cite[Lemma 2.3]{NSV} it suffices (in the language of said $\mathcal{H}$-character triples) to prove that 
$$(G \rtimes \mathrm{Aut}(G)_{B,{\chi_0^{\mathcal{H}} }} , G, \chi_0)_{\mathcal{H}} \geq_c ( B \rtimes \mathrm{Aut}(G)_{B,\psi_0^{\mathcal{H}} }, B, \psi_0)_{\mathcal{H}}.$$
According to \cite[Remark 3.4(a)]{Spaeth}, the group $\tilde{G} D$ induces all automorphisms of $\mathrm{Aut}(G)$. Hence, by \cite[Theorem 2.9]{NSV} it is enough to prove that
$$((\tilde{G}D)_{\chi_0^{\mathcal{H}} }, G, \chi_0)_{\mathcal{H}} \geq_c ( (\tilde{B}D)_{\psi_0^{\mathcal{H}} }, B, \psi_0)_{\mathcal{H}}.$$
According to Lemma \ref{field of value}, for every $\sigma \in \mathcal{G}$ there exists some $\tilde{t} \in \tilde{T}^{D}$ with $\phi_S^{\tilde{t}}= \phi_S^\sigma$ for all $S \subseteq \{ 1, \dots, r \}$. Therefore, Assumption \ref{tilde} is satisfied. We let $\mathcal{D}_2$ be a representation affording the extension of $\chi_0$ to $G D_{\chi_0}$ constructed in Proposition \ref{invariant extension}. We let $\mu_2 \in \Irr(D_{\chi_0})$ be the linear character such that $\chi_2^{x_\sigma}=\chi_2 \mu_2$. (In fact since $\tilde{t}$ is $D$-stable we have $\mu_2=1_{D_{\chi_0}}$ however we will not use this fact.) Furthermore, let $\mathcal{D}_1$ be a representation of $\tilde{G}_{\chi_0}$  affording the Clifford correspondent $\chi_1 \in \Irr(\tilde{G}_{\chi_0} \mid \chi_0 )$ of $\chi$. Similarly, let $\mathcal{D}'_2$ be a representation affording the extension of $\psi_0$ to $B D_{\psi_0}$ constructed in Proposition \ref{invariant extension local} and $\mathcal{D}'_1$ a representation of $\tilde{B}_{\psi_0}$ affording the unique character $\psi_1 \in \Irr(\tilde{B}_{\psi_0} \mid \psi_0 )$ with $\psi_1^{\tilde{B}}=\psi$. As in \ref{proj rep}, we let $\mathcal{P}$ and $\mathcal{P}'$ be the projective representations of $(\tilde{G}D)_{\chi_0}$ and $(\tilde{B} D)_{\psi_0}$ constructed with the linear representations $\mathcal{D}_i$ and $\mathcal{D}'_i$, $i=1,2$, respectively.

According to \cite[Theorem 1.1]{Spaeth} it suffices to prove that for every $a \in (\mathcal{H} \times \mathcal{B})_{\chi_0}$, the functions $\mu_a$
and $\mu'_a$ given by Lemma \ref{function} agree on $(\mathcal{H} \times \mathcal{B})_{\chi_0}$. By Lemma \ref{global} there exists some $\sigma \in \mathcal{H}$ such that $y:=a x_\sigma^{-1} \in (\tilde{B}D)_{\chi_0}$, where $x_\sigma=F_{p^e} \sigma^{-1} \tilde{t}$. By Lemma \ref{field of value}, the element $F_{p^e} \tilde{t}$ stabilizes $G D_{\chi_0}$. Let $\mu_1 \in \Irr(\tilde{G}_{\chi_0} / G)$ such that $ \chi_1^{x_\sigma} = \mu_1 \chi_1$. By Lemma \ref{compute mu} and Corollary \ref{invariant global} there exists an invertible matrix $M$ such that
$$\mathcal{P}^{x_\sigma}( \tilde{g} d)= \mu_1(\tilde{g}) \mu_2(d)  M \mathcal{P}(\tilde{g} d) M^{-1}$$ for all $\tilde{g} \in \tilde{G}_{\chi_0}$ and $d \in G D_{\chi_0}$. Since $y \in (\tilde{B} D)_{\chi_0}$ we obtain
$$ \mathcal{P}^y( \tilde{g} d)=  \mathcal{P}(y) \mathcal{P}( \tilde{g} d) \mathcal{P}(y)^{-1}$$
by \cite[Lemma 10.10(a)]{NavarroBook}. Therefore, we have
$$ \mathcal{P}^a(\tilde{g} d)= \mu_1^y( \tilde{g}) \mu_2^y(d) S \mathcal{P}(\tilde{g}d )S^{-1}$$
for the invertible matrix $S:=M \mathcal{P}(y)$. This shows that $\mu_a(\tilde{g} d)= \mu_1^y(\tilde{g}) \mu_2^y(d)$. By \cite[Theorem 3.5(b)]{Spaeth} or \cite[Corollary 3.20]{Ruhstorfer} it follows that $ \psi_1^{x_\sigma} = (\mu_1)_{\tilde{B}_{\psi_0}} \psi_1$.  By Lemma \ref{compute mu} and Proposition \ref{invariant extension local} there exists an invertible matrix $M'$ such that
$$\mathcal{P}'^{x_\sigma}( \tilde{g} d)= \mu_1(\tilde{g})  M' \mathcal{P}'(\tilde{g} d) M'^{-1}$$ for all $\tilde{g} \in \tilde{B}_\psi$ and $d \in B D_{\psi_0}$. The same calculation as in the global case now shows that $\mu'_a(\tilde{g} d)= \mu_1^y( \tilde{g}) \mu_2^y(d)$ for $\tilde{g}d \in (\tilde{B}D)_{\psi_0}$. 
Consequently, the functions $\mu_a$
and $\mu'_a$ agree on $(\mathcal{H} \times \mathcal{B})_{\psi_0}$.
%
%
% Since $\tilde{f}$ is $\mathcal{H} \times \mathcal{B}$-equivariant it follows that $\mu'_a( \tilde{g} d)= \mu_1({}^y \tilde{g})$. for every $a \in (\mathcal{H} \times \mathcal{B})_{\psi_0}$ and $\tilde{g} d \in (\tilde{B} D)_{\psi_0}$.
\end{proof}

%\begin{remark}
%	Suppose that the same assumptions as in Theorem \ref{main2} hold but suppose this time that $F=F_q \gamma$ is a twisted Frobenius endomorphism. Then using Remark \ref{twisted} we observe that $S$ satisfies the inductive Galois--McKay condition provided that $q$ is not a square.
%	\end{remark}

\subsection{An alternative approach to the inductive condition}

In the statement of Theorem \ref{main2} we needed to exclude some cases. This was essentially because the statement of Lemma \ref{field of value}(a) doesn't hold for these groups in general. However, using one of the main results of \cite{TiepZalesski} and the strategy of \cite{farrell2019fake} we can prove the inductive Galois--McKay condition in the remaining cases. The proof also highlights some of the difficulties when dealing with the inductive Galois--McKay condition instead of the inductive McKay condition.

\begin{theorem}\label{part two}
	Suppose that $\G$ is of type $A_n$ or $D_n$ for odd $n$. Assume that $F: \G \to \G$ is a Frobenius endomorphism such that $S:=\G^F/ \mathrm{Z}(\G^F)$ is a simple non-abelian group and $\G^F$ is its universal covering group. Then the inductive Galois--McKay condition from \cite[Definition 3.1]{NSV} holds for $S$ and $p$.
\end{theorem}

\begin{proof}
%	By Theorem \ref{main} there exists an $\mathcal{H} \times \mathcal{B}$-equivariant bijection $f:\Irr_{p'}(B)\to \Irr_{p'}(G)$. Let $\chi_0 \in \mathcal{T}_0$ be arbitrary and set $\psi_0:=f(\chi_0)$. We fix a character $\chi \in \Irr(\tilde{G})$ lying above $\chi_0$ and denote $\psi:=\tilde{f}( \chi)$. Furthermore, we let $S \subseteq \{1,\dots, r\}$ be such that $\psi_0 \in \Irr(B \mid \phi_S)$.
	We assume the notation of the proof of Theorem \ref{main2}. Arguing as in the proof of loc. cit. it suffices to prove that
	$$(G \rtimes \mathrm{Aut}(G)_{B,{\chi_0^{\mathcal{H}} }} , G, \chi_0)_{\mathcal{H}} \geq_c ( B \rtimes \mathrm{Aut}(G)_{B,\psi_0^{\mathcal{H}} }, B, \psi_0)_{\mathcal{H}}.$$
	If for every $\sigma \in \mathcal{H}$ there exists $\tilde{t} \in \tilde{T}^{D_{\chi_0}}$ such that $\xi^\sigma= \xi^{\tilde{t}}$ then this follows from the arguments given in Theorem \ref{main2}. We will now discuss which characters $\chi_0 \in \Irr_{p'}(G)$ are not covered by this argument. Note that by Lemma \ref{field of value}(b) we can assume that $p \neq 2$. Furthermore, if $q$ is a square then we have $\xi^\sigma= \xi$ by \cite[Theorem 1.8(i)]{TiepZalesski} and therefore we can also exclude this case.
 %	Firstly, by Lemma \ref{field of value}(b) such an element $\tilde{t} \in \tilde{T}^{D_{\chi_0}}$ exists whenever $S$ is a proper subset of $\{1,\dots, r\}$. We may therefore assume that $S= \{1, \dots, r \}$.
% By Lemma \ref{field of value}(a) we can assume that $p \neq 2$.
	
	Suppose first that $D_{\chi_0}=\langle F' \rangle$ for some Frobenius endomorphism $F': \Gtilde \to \Gtilde$ which satisfies $(F')^s=F$ for some $s \geq 1$. By Lemma \ref{field of value: better} there exists some $\tilde{t} \in \tilde{T}$ such that $\xi^\sigma= \xi^{\tilde{t}}$ and $F'(\tilde{t}) \tilde{t}^{-1} \in \mathrm{Z}(\Gtilde^F)$. Applying Lang's theorem to the Frobenius map $F': \mathrm{Z}(\Gtilde) \to \mathrm{Z}(\Gtilde)$ we find $z \in \mathrm{Z}(\Gtilde)$ such that $\tilde{t}':=z \tilde{t} \in \Ttilde^{F'}= \tilde{T}^{D_{\chi_0}}$. Since we still have $\xi^\sigma= \xi^{\tilde{t}'}$ we see that the arguments of Theorem \ref{main2} also apply in this case.
	
Assume therefore now that no such Frobenius endomorphism exists. Since $D= \langle F_p, \gamma \rangle$ we can therefore assume that $\gamma \in D_{\chi_0}$, where $\gamma$ is the graph automorphism of order $2$ of $\G$. In the twisted case, this only works because we can assume that $q$ is not a square.
	 We let $\theta \in \Irr( \mathrm{Z}(G) \mid \chi_0)= \Irr( \mathrm{Z}(G) \mid \psi_0)$. Since $\theta$ is $\gamma$-stable it follows that $\theta \in \Irr( \mathrm{Z}(G)) $ satisfies $\theta^2=1_{\mathrm{Z}(G)}$.
	
%If $\G^F=A_n(q)$ for odd $n$ then by \cite[Theorem 1.8 (ii)]{TiepZalesski} we have $\phi_S^\sigma= \phi_S$ for all $S \subseteq \{1, \dots, r \}$ if $\frac{n+1}{(n+1,q- 1)}$ is even. This case is then also covered by the arguments in Theorem \ref{main2} so we may assume that $$|\mathrm{Z}(\G)/ \mathrm{Z}(\G)^F|=\frac{(n+1)_{p'} }{(n+1,q -  1)}$$ is odd.

%In particular, in both cases we can and we will assume that $|\mathrm{Z}(\G)/ \mathrm{Z}(\G)^F|$ is odd.

%	By assumption $\G^F=A_n(\varepsilon q)$ for odd $n$. Thus, by \cite[Theorem 1.8 (ii)]{TiepZalesski} we have $\phi_S^\sigma= \phi_S$ for all $S \subseteq \{1, \dots, r \}$ if $\frac{n+1}{(n+1,q- \varepsilon)}$ is even. This case is then also covered by the arguments in Theorem \ref{main2} so we may assume that $$|\mathrm{Z}(\G)/ \mathrm{Z}(\G)^F|=\frac{(n+1)_{p'} }{(n+1,q -  \varepsilon)}$$ is odd.
	
	We will now recall the method introduced in \cite{farrell2019fake}. We denote by $\mathcal{L}: \Gtilde \to \Gtilde, \, g \mapsto g^{-1} F(g)$ the Lang map on $\Gtilde$. Let $E \subseteq \mathrm{Aut}(\mathcal{L}^{-1}( \mathrm{Z}(\G)))$ be the subgroup generated by $F_p$ and $\gamma$. The group $E$ acts by automorphisms on $\tilde{G}$ and so we can form the semi-direct product $\tilde{G} \rtimes E$, which generates all automorphisms of $G$. By \cite[Theorem 2.9]{NSV} we can equivalently prove that
	$$((\tilde{G} E)_{\chi_0^{\mathcal{H}} }, G, \chi_0)_{\mathcal{H}} \geq_c ( (\tilde{B}             E)_{\psi_0^{\mathcal{H}} }, B, \psi_0)_{\mathcal{H}}.$$
%Since the character $\chi_0$ has label $(c_0,0,\cdots,0)$ for some $c_0 \in \mathbb{F}_q^\times$ it follows that $\tilde{G}_{\chi_0}=G \mathrm{Z}(G)$. From this we deduce that
%	$$(\mathcal{L}^{-1}( \mathrm{Z}(\G)) E)_{\chi_0} = \mathrm{Z}(\G) E_{\chi_0}.$$
	By Lemma \ref{field of value: better} there exists some $\tilde{t} \in \tilde{T}$ with $d(\tilde{t}) \tilde{t}^{-1} \in \mathrm{Z}(\tilde{G})$ for all $d \in D$ and $\xi^\sigma= \xi^{\tilde{t}}$.
	 Observe that $\gamma$ acts by inversion on $\mathrm{Z}(\G)$ and hence also on $H^1(F, \mathrm{Z}(\G))$. We therefore have
	$$\tilde{t}_2 \tilde{t}_{2'}= \tilde{t}=\gamma(\tilde{t})= \tilde{t}_2^{-1} \tilde{t}_{2'}^{-1}$$
	in $H^1(F, \mathrm{Z}(\G))$. This shows that $(t_{2'})^2=1$ in $H^1(F, \mathrm{Z}(\G))$ and so $t_{2'}=1$ in $H^1(F,\mathrm{Z}(\G))$. By replacing $\tilde{t}$ with its $2$-part we may assume that $\tilde{t}$ is an element of $2$-power order.
	
	 Since $\tilde{\T}= \T \mathrm{Z}(\tilde{\G})$ we find some $z \in \mathrm{Z}(\tilde{\G})$, which we can assume to be of $2$-power order, such that $t:=z \tilde{t} \in \T$. In particular, $t$ is a $2$-element again. We observe that $\mathcal{L}(t)= \mathcal{L}(z) \in \mathrm{Z}(\tilde{\G}) \cap \G= \mathrm{Z}(\G)$ and $a(t) t^{-1} \in \mathrm{Z}(\G)$ for all $a \in E$. Furthermore, we have $\xi^\sigma= \xi^{t}$. We set $y_\sigma:=F_{p^e} \sigma^{-1} t \in \mathcal{L}^{-1}(\mathrm{Z}(\G)) E \times \mathcal{H}$ and $\pi:=F_{p^e} \sigma^{-1}$. We will now make the following important observation.

%\begin{lemma}\label{hilf1}
%	$\mathcal{L}(t)^2 \in \mathcal{L}(\mathrm{Z}(\G))$.
%\end{lemma}
%	
%\begin{proof}
%	We have $\gamma(t)t^{-1} \in \mathrm{Z}(\G)$ and so $\mathcal{L}(\gamma(t)t^{-1}) \in \mathcal{L}(\mathrm{Z}(G))$. Let $\gamma$ be the non-trivial graph automorphism. Then $$\mathcal{L}(\gamma(t)t^{-1})=\gamma(\mathcal{L}(t)) \mathcal{L}(t)^{-1} = \mathcal{L}(t)^{-2},$$
%	where the last equality follows from the fact that $\gamma$ acts by inversion on $\mathrm{Z}(\G)$.
%\end{proof}
%
\begin{lemma}\label{hilf2}
To prove Theorem \ref{part two}	we can assume that $\mathcal{L}(\mathrm{Z}(\G))_2 \subseteq \mathrm{Z}(G)_2$. 
\end{lemma}

\begin{proof}
	Suppose that $\G^F=D_n( \varepsilon q)$ for odd $n$. If $|\mathrm{Z}(\G^F)| \in \{1,2 \}$ then we have $\phi_S^\sigma= \phi_S$ for all $S \subseteq \{1, \dots, r\}$ by \cite[Theorem 1.8(iv)]{TiepZalesski}. Hence, this case is already covered by the arguments in Theorem \ref{main2} and we can even assume that $\mathrm{Z}(\G)= \mathrm{Z}(\G^F)$.

Suppose now that $\G$ is of type $A$. The Lang map $\mathcal{L}: \mathrm{Z}(\G) \to \mathrm{Z}(\G)$ yields an isomorphism $\mathcal{L}(\mathrm{Z}(\G)) \cong \mathrm{Z}(\G) / \mathrm{Z}(G)$. Since $\mathrm{Z}(\G)$ is cyclic the claim is hence equivalent to $|\mathcal{L}(\mathrm{Z}(\G))|_2 \leq |\mathrm{Z}(G)|_2$. This is equivalent (since we assume that $p \neq 2$) to $\frac{n_2}{(n,q - \varepsilon )_2} \leq (n,q-\varepsilon)_2$. This is satisfied whenever $n_2 \leq (q-\varepsilon)_2$. If we assume that $n_2 > (q-\varepsilon)_2$ then we have $\mathrm{Z}(\tilde{G})_2 \subseteq \mathrm{Z}(G)$. Since $\tilde{t}$ is a $2$-element we conclude that $\gamma( \tilde{t}) \tilde{t}^{-1} \in \mathrm{Z}(\tilde{G})_2 \subseteq G$. Consequently, $d(\tilde{t}) \tilde{t}^{-1} \in \mathrm{Z}(G)$ for all $d \in D$. In this case the proof of Theorem \ref{main2} applies and we can conclude that the inductive Galois-McKay condition holds in that case.
%We have $\mathrm{det}(\gamma(\tilde{t}) \tilde{t}^{-1})= \mathrm{det}(\gamma(\tilde{t})) \mathrm{det}( \tilde{t}^{-1})=\mathrm{det}(\tilde{t})^{-2}$
\end{proof}
	
%	By Lemma \ref{hilf1} we have $\mathcal{L}(t)^2 \in \mathcal{L}(\mathrm{Z}(\G))$. Hence $\mathcal{L}(t_2) \in \mathcal{L}(\mathrm{Z}(\G))$.
 	
The previous lemma implies that the element $t$ normalizes $BE$. Indeed, for $a \in E$ we have $\mathcal{L}(a(t) t^{-1}) \in \mathcal{L}(\mathrm{Z}(\G))_2 \subseteq \mathrm{Z}(G)$ and therefore $E^t=E$.

%	The character $\theta \in \Irr( \mathrm{Z}(G) \mid \psi_0)$ has order $2$. Since $|\mathrm{Z}(\G)/ \mathrm{Z}(\G)^F|$ is odd, there exists a unique character $\hat{\theta} \in \Irr(\mathrm{Z}(\G) \mid \theta)$ of order $2$. Consequently, $\hat{\theta}$ is $E$-stable.
We will now first consider the local situation. Let $\hat{\lambda} \in \Irr(B_{\phi_S} E_{\chi_0} )$ be the unique extension of $\phi_S$ with $E_{\chi_0} \subseteq \mathrm{Ker}(\hat{\lambda})$. By Mackey's formula the character $\hat{\psi}_2:= \hat{\lambda}^{B E_{\chi_0} }$ is an extension of $\psi_2$. Since $\phi_S$ is $y_\sigma$-stable, we have $\hat{\lambda}^{y_\sigma}= \mu \hat{\lambda}$ for some linear character $\mu \in \Irr(E_{\chi_0})$. Therefore, for $a \in E_{\chi_0}$ we have 
%		$$\hat{\lambda}^{y_\sigma}(e)=\hat{\lambda}^\pi( e(t) t^{-1}) \hat{\lambda}^\pi(e)= \hat{\lambda}^\pi(e(t)t^{-1}).$$
%		We conclude that $\hat{\lambda}^{y_\sigma}= \mu \hat{\lambda}$, where $\mu \in \Irr(E_{\chi_0})$ is the linear character with values 
		$$\mu(a)=  \hat{\lambda}^\pi(a(t)t^{-1})= \theta^\pi(a(t) t^{-1}).$$
		In particular, we have $\hat{\psi}_2^{y_\sigma}=\hat{\psi}_2 \mu$. Since $\theta^2=1_{\mathrm{Z}(G)}$, we deduce that the linear character $\mu$ satisfies $\mu^2=1_{E_{\chi_0}}$. Furthermore, this implies $\mu(a)=\theta(a(t) t^{-1})$.

	Let us now consider the global situation. We can write $E_{\chi_0}=\langle F_0, \gamma \rangle$, where $F_0: \G \to \G$ is a Frobenius endomorphism such that some power of $F_0$ is $F$ and $F_0(\tilde{t})=\tilde{t}$. This is always possible since $\tilde{t}$ is $F_p \gamma^{\frac{1-\varepsilon}{2}}$-stable by the proof of Lemma \ref{field of value: better} and $(F_p \gamma^{\frac{1-\varepsilon}{2}})^f=F$ since $q$ is odd. Consider the Gelfand--Graev character $\hat{\Gamma}_1:=\hat{\xi}^{ G \langle \gamma \rangle }$, where $\hat{\xi} \in \Irr(U \langle \gamma \rangle )$ is the unique extension of $\xi$ with $\hat{\xi}(\gamma)=1$. By the above calculations, we have $\hat{\xi}^{y_\sigma}=\hat{\xi} \mu$ and so $\hat{\Gamma}_1^{y_\sigma}= \mu \hat{\Gamma}_1$.
%	Note that $p \neq 2$ so that $\mu$ is a character of $p'$-order.
	We can consider $F_{p^{e}} t$ as a bijective morphism of $\G \rtimes \langle \gamma \rangle$ which stabilizes $\T \langle \gamma \rangle $. Since $\gamma(t)t^{-1} \in \mathrm{Z}(\G)$ it follows that the bijective morphism $F_{p^e}t: \G\langle \gamma \rangle \to \G \langle \gamma \rangle$ commutes with the Frobenius endomorphism $F$ of $\G \rtimes \langle \gamma \rangle$ and stabilizes $\T \langle \gamma \rangle$. By Proposition \ref{invariantextension} we obtain an extension $\chi_1 \in \Irr(G \langle \gamma \rangle)$ of $\chi_0$ to $\G^F \langle \gamma \rangle$, such that $\chi_1^{y_\sigma}= \mu \chi_1$.
	
	Denote $x_\sigma:=F_{p^e} \sigma^{-1} \tilde{t} \in \tilde{G}D \times \mathcal{H}$. Since $\tilde{t}$ is $F_0$-stable we obtain by Proposition \ref{invariant extension} an extension $\overline{\chi}_2$ of $\chi$ to $G \langle F_0|_{\G^F} \rangle$ which is $x_\sigma$-stable. Using the group epimorphism $G \langle F_0 \rangle \to G \langle F_0|_{\G^F} \rangle$ we can lift $\overline{\chi}_2$ to an extension $\chi_2$ of $G \langle F_0 \rangle$, with $F \in E$ in its kernel.
Since $y_\sigma=x_\sigma z$ we deduce that $\chi_2^{y_\sigma}=\chi_2^{z}$. For $x=ga$ with $g \in G$ and $a \in \langle F_0 \rangle$, we have
$$\chi_2^z(x)=\chi_2(a(z) z^{-1} x)= \theta(a(z) z^{-1} ) \chi_2(x)=\mu(a) \chi_2(x)$$
or in other words $\chi_2^{y_\sigma}= \mu \chi_2$. Now Lemma \ref{direct product} shows the existence of an extension ${\hat{\chi}}_2 \in \Irr(G E \mid \chi_0)$ which satisfies ${\hat{\chi}}_2^{y_\sigma}=\mu {\hat{\chi}}_2$.
	
	We are now ready to verify the inductive conditions. This is proved in a similar fashion as in Theorem \ref{main2}. Let $X \in \{ B, G \}$ and 
	$\vartheta_0 \in \{ \psi_0, \chi_0 \}$ be the corresponding character of $X$. We let $\mathcal{D}$ be a representation of $X$ affording $\vartheta_0$. There exists a representation $\mathcal{D}_1$ of $X_1:=\tilde{G}_{\chi_0} \cap X$, extending $\mathcal{D}$. We denote by $\hat{\vartheta}_1 \in \Irr(X_1/X)$ its character. Let $\mathcal{D}_2$ be a representation of $X_{2}:=X E_{\vartheta_0}$ affording the character $\hat{\vartheta}_2 \in \{\hat{\psi}_2 , \hat{\chi}_2 \}$ which extends the representation $\mathcal{D}$.
	 We consider the projective representation $\mathcal{P}$ of $X_1 X_2$ defined by 
	$$\mathcal{P}(x_1 x_2):= \mathcal{D}_1(x_1) \mathcal{D}_2(x_2)$$
	for $x_1 \in X_1$ and $x_2 \in X_2$.

%			According to \cite[Theorem 1.1]{Spaeth} it suffices to prove that for every $a \in (\mathcal{H} \times \mathcal{B})_{\chi_0}$, the functions $\mu_a$
%			and $\mu'_a$ given by Lemma \ref{function} agree on $(\mathcal{H} \times \mathcal{B})_{\chi_0}$.
			
 Let $\mu_1 \in \Irr(X_1 / X)$ such that $\hat{\chi}_1^{y_\sigma} = \mu_1 \hat{\chi}_1$.
Note that by \cite[Theorem 3.5(b)]{Spaeth} or \cite[Corollary 3.20]{Ruhstorfer} it follows that $ \hat{\psi}_1^{y_\sigma} = \mu_1 \psi_1$. We have $\mathcal{P}|_{\mathrm{Z}(\G) \langle F \rangle}= \theta \times 1_{\langle F \rangle} E_{\hat{\vartheta}_0(1)}$. Now \cite[Lemma 2.7]{farrell2019fake} implies that
$$((\tilde{G} E)_{\chi_0}, G, \chi_0 ) \geq_c ( (\tilde{B}              E)_{\psi_0} , B, \psi_0).$$
 We now show that the functions $\mu_a$ and $\mu_a'$ given by Lemma  \ref{function}
 for the representations $\mathcal{P}$ agree on $(\mathcal{H} \times X_1 X_2)_{\psi_0}$. By Lemma \ref{global} there exists some $\sigma \in \mathcal{H}$ such that $y:=a y_\sigma^{-1} \in X_1 X_2$.
  By Lemma \ref{compute mu} there exists an invertible matrix $M$ such that
			$$\mathcal{P}^{y_\sigma}( x_1 x_2)= \mu_1(x_1) \mu_2(x_2)  M \mathcal{P}(x_1 x_2) M^{-1}$$ for all $x_1 \in X_1$ and $x_2 \in X_2$. Since $y \in X_1 X_2$ we obtain
			$$ \mathcal{P}^y( x_1 x_2)=  \mathcal{P}(y) \mathcal{P}( x_1 x_2) \mathcal{P}(y)^{-1}$$
			by \cite[Lemma 10.10(a)]{NavarroBook}. Therefore, we have
			$$ \mathcal{P}^a(x_1 x_2)= \mu_1^y( x_1) \mu_2^y(x_2) S \mathcal{P}(\tilde{g}d )S^{-1}$$
			for the invertible matrix $S:=M \mathcal{P}(y)$. This shows that $\mu_a(x_1 x_2)= \mu_1^y( x_1) \mu_2^y(x_2)$. Therefore, the function $\mu_a$ is the same in the local and the global case. Consequently, we have
				$$((\tilde{G} E)_{\chi_0^{\mathcal{H}} }, G, \chi_0)_{\mathcal{H}} \geq_c ( (\tilde{B}             E)_{\psi_0^{\mathcal{H}} }, B, \psi_0)_{\mathcal{H}},$$
	 which finishes the proof.
%	Therefore, ${}^z  \chi_2=  \chi_2$ and since $\tilde{t}=t z$ we deduce that $\chi_2$ is $y_\sigma$-stable as well. Corollary \ref{invariant global} now yields an extension of $\chi_0$ to $G D_{\chi_0}$ which is $y_\sigma$-stable.
%	
\end{proof}

\end{document}